\documentclass[11pt,reqno]{amsart}
\usepackage{amsmath,amssymb,mathrsfs}
\usepackage{graphicx,cite}
\usepackage{subfigure}

\newtheorem{theorem}{Theorem}[section]

\newtheorem{lemma}[theorem]{Lemma}

\newtheorem{remark}[theorem]{Remark}

\newtheorem{assumption}{Assumption}
\newtheorem{definition}{Definition}

\numberwithin{equation}{section}

\allowdisplaybreaks[4]

\begin{document}

\title[An inverse random source problem]{An inverse random source problem for the biharmonic wave equation}

\author{Peijun Li}
\address{Department of Mathematics, Purdue University, West Lafayette, Indiana 47907, USA}
\email{lipeijun@math.purdue.edu}

\author{Xu Wang}
\address{Department of Mathematics, Purdue University, West Lafayette, Indiana 47907, USA}
\email{wang4191@purdue.edu}

\thanks{The research is supported in by part the NSF grant DMS-1912704.}

\subjclass[2010]{35R30, 35R60, 65M32}

\keywords{inverse random source problem, biharmonic operator, Gaussian random fields, stochastic differential equations, pseudo-differential operator, principal symbol}

\begin{abstract}
This paper is concerned with an inverse source problem for the stochastic biharmonic operator wave equation. The driven source is assumed to be a microlocally isotropic Gaussian random field with its covariance operator being a classical pseudo-differential operator. The well-posedness of the direct problem is examined in the distribution sense and the regularity of the solution is discussed for the given rough source. For the inverse problem, the strength of the random source, involved in the principal symbol of its covariance operator, is shown to be uniquely determined by a single realization of the magnitude of the wave field averaged over the frequency band with probability one. Numerical experiments are presented to illustrate the validity and effectiveness of the proposed method for the case that the random source is the white noise. 
\end{abstract}

\maketitle

\section{Introduction}

As one of the important research subjects in inverse scattering theory, inverse source problems for wave propagation have diverse scientific and industrial applications such as antenna design and synthesis, medical imaging \cite{I90}. They have continuously attracted much attention from many researchers. We refer to \cite{BLLT15} and the references cited therein for some recent advances on this topic. Meanwhile, the study on boundary value problems for higher-order elliptic operators has generated sustained interest in the mathematics community \cite{GGS10}. The biharmonic operator, which may arise from the modeling of elasticity for example, appears to be a natural candidate for such a study \cite{M3P09, RR20, S00}. Compared with inverse problems involving the second order differential operators, the inverse problems for the biharmonic operator are much less studied. The reason is not only the increase of the order which leads to the failure of the methods developed for the second order equations, but also the properties of the solutions for the higher order equations are more sophisticated. Some of the inverse boundary value problems for bi- and poly-harmonic operators can be found in \cite{GX19, KLG12, KLG14, LYZ, TH17, TS18, Y14}. 

In practice, there are many uncertainties caused by the unpredictability of the surrounding environment, incomplete knowledge of the studied system, fine-scale spatial or temporal variations, etc., which cannot be neglected during analysis or simulation. To take account of uncertainties, it would be reasonable and important to introduce random parameters to the mathematical modeling. Stochastic inverse problems refer to as inverse problems that involve randomness. Compared to their deterministic counterparts, stochastic inverse problems are more difficult due to two extra challenges: the random parameter is sometimes too rough to exist point-wisely and can only be interpreted as a distribution; the statistics such as the average and variance of the random parameter are required to be reconstructed. New methodology needs to be developed not only for the inverse problems but also for the corresponding direct problems in stochastic settings. 

In this work, we consider an inverse source problem for the stochastic biharmonic wave equation
\begin{equation}\label{eq:model}
\Delta^2u-k^4u=f\quad \text{in}~\mathbb R^d,
\end{equation}
where $d=2$ or $3$ and $k>0$ is the wavenumber.  The wave field $u$ and its Laplacian $\Delta u$ are required to satisfy the Sommerfeld radiation condition
\begin{equation}\label{eq:radiation}
\lim_{r\to\infty}r^{\frac{d-1}2}\left(\partial_r u-{\rm i}ku\right)=\lim_{r\to\infty}r^{\frac{d-1}2}\left(
\partial_r\Delta u-{\rm i}k\Delta u\right)=0,\quad r=|x|.
\end{equation}
The source $f$ is assumed to be a microlocally isotropic Gaussian random field of order $-m$ (cf. Definition \ref{def})
such that its covariance operator is a classical pseudo-differential operator with principal symbol $\mu(x)|\xi|^{-m}$, where $\mu$ is called the strength of the random source $f$. The microlocally isotropic Gaussian random field can be viewed as one of the generalized fractional Gaussian random fields (cf. \cite{LW21}), which cover a wide class of frequently studied Gaussian random fields, such as the white noise with $m=0$ and translations of the classical fractional Brownian motions with $m\in(d,d+1)$. In particular, if $m\le d$, the random field $f$ is too rough to exist point-wisely, and should be interpreted as a distribution.

For the white noise case with $m=0$, the random source can be equivalently rewritten as $f=\sqrt{\mu}\dot{W}$, where $\dot{W}$ denotes the white noise. Then the biharmonic wave equation \eqref{eq:model} is interpreted as a stochastic partial differential equation driven by an additive white noise. The It\^o isometry can be used in this case to derive the recovery formula for the strength $\mu$. We refer to \cite{BCL16,BCLZ14} and \cite{BCL17} for the inverse random source problem of the acoustic and elastic wave equations, respectively, where the strength $\mu$ is shown to be uniquely determined by the variance of the wave field at multiple frequencies.

As a generalized Gaussian random field, the microlocally isotropic Gaussian random field with a general $m$ is studied in recent years (cf. \cite{CHL19,HLO14,LHL20,LW21}) to handle a larger class of Gaussian random fields whose increments are not independent if $m\neq0$ and hence the It\^o isometry is not available. For the case $m\in[d,d+\frac12)$, by using the microlocal analysis of the Fourier integral operators, it was shown in \cite{LHL20} for both the acoustic and elastic wave equations that the strength $\mu$ is uniquely determined by almost surely a single realization of the amplitude of the scattering field averaged over the frequency band. In \cite{LW21} and \cite{LW21c}, these results are extended to rougher sources with $m\in(d-2,d]$ for the acoustic and electromagnetic wave equations by exploring an equivalent model in terms of the fractional Laplacian $\sqrt{\mu}(-\Delta)^{-\frac m4}\dot{W}$. We mention that the existing work do not contain the case $m=0$ for $d=2,3$, i.e., the white noise case is not included in the framework of the study for microlocally isotropic Gaussian random fields. To the best of our knowledge, little is known for stochastic inverse problems on higher order wave equations. This is the first study on the inverse random source problem of the biharmonic operator wave equation. 

In this paper, we intend to examine both the direct and inverse source problems for the biharmonic operator. A particular interest is on the rough source with $m\le d$ such that $f$ should be interpreted as a distribution. 
We show that the direct problem is well-posed with $m\in(d-6,d]$ in the distribution sense (cf. Theorem \ref{tm:wellposed}). The results of this work contain the white noise case $m=0$ and even rougher cases $m<0$ for both the two- and three-dimensional problems due to the fact that the fundamental solution to the biharmonic operator is more regular than that of the Helmholtz operator (cf. Lemma \ref{lm:Hk}).  For the inverse problem, we prove that the strength $\mu$ of the random source is uniquely determined by almost surely a single realization of the magnitude of the wave field $u$ averaged over the frequency band (cf. Theorems \ref{tm:u3d} and \ref{tm:u2d}), which is summarized in the following theorem.

\begin{theorem}
Let $f$ be a centered microlocally isotropic Gaussian random field of order $-m$ in a bounded domain $D\subset\mathbb R^d$ with $m\in(d-6,d]$ and $d=2,3$, and $U\subset\mathbb R^d$ be a bounded domain that having a positive distance to $D$, i.e., dist$(D,U)=r_0>0$. For any $x\in U$, it holds almost surely that
\[
\lim_{T\to\infty}\frac1T\int_T^{2T}k^{m+7-d}|u(x;k)|^2dk=\frac1{16(2\pi)^{d-1}}\int_D\frac1{|x-\zeta|^{d-1}}\mu(\zeta)d\zeta=:T_d(x).
\]
Moreover, the strength $\mu$ can be uniquely determined by data $\{T_d(x)\}_{x\in U}$.
\end{theorem}

The paper is organized as follows. In Section \ref{sec:pre}, we introduce the regularity and kernel functions of microlocally isotropic Gaussian random fields, as well as the fundamental solution to the biharmonic operator wave equation. Section \ref{sec:direct} addresses the well-posedness of the direct problem and the regularity of the solution for the stochastic biharmonic wave equation. Section \ref{sec:inverse} is devoted to the inverse problem, where the uniqueness is obtained for the reconstruction of the strength of the random source. Numerical experiments are presented in Section \ref{sec:num} for the white noise case to illustrate the theoretical results. The paper is concluded with some general remarks and future work in Section \ref{co}. 

\section{Preliminaries}
\label{sec:pre}

In this section, we introduce some basic properties of microlocally isotropic Gaussian random fields and the fundamental solution to the biharmonic operator wave equation, which are essential for the study of both the direct and inverse problems. 

\subsection{Microlocally isotropic Gaussian random fields}

Let us begin with the definition of a microlocally isotropic Gaussian random field, and then we discuss the regularity and the kernel function of such a random field.  

\begin{definition}\label{def}
A Gaussian random field $f$ is said to be microlocally isotropic of order $-m$ in $D\subset\mathbb R^d$ if its covariance operator $\mathcal Q_f$ is a classical pseudo-differential operator and the principal symbol of $\mathcal Q_f$ has the form $\mu(x)|\xi|^{-m}$ with $\mu\in C_0^{\infty}(D)$ and $\mu\ge0$, where $\mu$ is called the strength of the random field $f$.
\end{definition}

As is known, a pseudo-differential operator can be expressed through the Fourier transform
\begin{align}\label{eq:Qf1}
(\mathcal Q_f\varphi)(x):=\frac1{(2\pi)^d}\int_{\mathbb R^d}e^{{\rm i}x\cdot\xi}\sigma(x,\xi)\hat\varphi(\xi)d\xi,
\end{align}
where $\sigma\in S^{-m}(\mathbb R^d\times\mathbb R^d)$ is called the symbol of the pseudo-differential operator. Here 
\[
S^{-m}(\mathbb R^d\times\mathbb R^d):=\Big\{a(z,\xi)\in C^{\infty}(\mathbb R^d\times\mathbb R^d): |\partial_\xi^\alpha\partial_z^\beta a(z,\xi)|\le C_{\alpha,\beta}(1+|\xi|)^{-m-|\alpha|}\Big\}
\]
is the space of symbols of order $-m$, where $\alpha,\beta$ are multi-indices whose length are defined by $|\alpha|:=\sum_{j=1}^d\alpha_j$ for any multi-index $\alpha=(\alpha_1,\cdots,\alpha_d)$.

The microlocally isotropic Gaussian random field covers a wide range of frequently studied Gaussian random fields such as the white noise and translated fraction Brownian motions (cf. \cite{LW21}). It possesses several important properties, which play an important role in the recovery of the strength for the random source.  For example, the symbol $\sigma$ of the covariance operator $\mathcal Q_f$ is invariant under changes of variables. Moreover, the Schwartz kernel $K_f$ given by 
\begin{align}\label{eq:Qf2}
 (\mathcal Q_f\varphi)(x)=\int_{\mathbb R^d}K_f(x,y)\varphi(y)dy
\end{align}
is a homogeneous function of $x-y$ and is singular only at the diagonal. Combining \eqref{eq:Qf1} and \eqref{eq:Qf2} yields that the kernel $K_f$ can be represented in terms of its symbol $\sigma$ via the Fourier transform (cf. \cite{LW21}): 
\begin{align*}
K_f(y,z)=\frac1{(2\pi)^d}\int_{\mathbb R^d}e^{{\rm i}(y-z)\cdot\xi}\sigma(z,\xi)d\xi.
\end{align*}

It is clear to note that the regularity of the random field $f$ is determined by its covariance operator $\mathcal Q_f$, and hence is determined essentially by the principal symbol of the pseudo-differential operator $\mathcal Q_f$. To investigate the regularity of $f$, we consider the following fractional Gaussian random field (cf. \cite{LW21,LSSW16}): 
\[
\tilde f:=\sqrt{\mu}(-\Delta)^{-\frac m4}\dot{W}, 
\]
where $\dot{W}$ denotes the white noise and can be understood as the formal derivative of the real-valued $d$-parameter Brownian sheet $W$ (cf. \cite[Chapter 2.1]{HOUZ10}). The regularity of $\tilde f$ is relatively easy to get since the regularity of the white noise has already been investigated. It is shown in \cite[Proposition 2.5]{LW21} that $\tilde f$ satisfies Assumption \ref{as:f} and has the principal symbol $\mu(x)|\xi|^{-m}$. Consequently, the microlocally isotropic Gaussian random field $f$ has the same regularity as $\tilde f$. The result is stated in the following lemma and the proof can be found in \cite[Lemma 2.6]{LW21}.

\begin{lemma}\label{lm:f}
Let $f$ be a microlocally isotropic Gaussian random field of order $-m$ in $D\subset\mathbb R^d$.
\begin{itemize}
\item[(i)] If $m\in(d,d+2)$, then $f\in C^{0, \alpha}(D)$ almost surely for all
$\alpha\in(0,\frac{m-d}2)$.
\item[(ii)] If $m\in(-\infty,d]$, then $f\in W^{\frac{m-d}2-\epsilon,p}(D)$ almost surely for all $\epsilon>0$ and $p>1$.
\end{itemize}
\end{lemma}

By Lemma \ref{lm:f}, if $m\in(d,d+2)$, then $f$ is almost surely H\"older continuous and is relatively smooth; 
if $m\in(-\infty,d]$, then the random field $f$ is too rough to exist point-wisely. For such a rough $f$, it should be interpreted as a distribution in the Schwartz distribution space $\mathcal D'$. The covariance operator is defined by
\[
\langle\mathcal Q_f\varphi,\psi\rangle:=\mathbb E[\langle f,\varphi\rangle\langle f,\psi\rangle]\quad\forall~\varphi,\psi\in\mathcal D,
\]
where $\mathcal D$ stands for the space of test functions with $\mathcal D'$ being its dual space, and 
\[
\langle f,\varphi\rangle:=\int_{\mathbb R^d}f(x)\varphi(x)dx
\] 
is the dual product. In this paper, we are interested in rough sources which satisfy following assumption.

\begin{assumption}\label{as:f}
 Assume that the random source $f$ is a  centered microlocally isotropic Gaussian random field of order $-m$ in a bounded domain $D\subset\mathbb R^d$ with strength $\mu$ and $m\in(d-6,d]$.
 \end{assumption}
 
According to the relationship between $f$ and $\tilde f$, the leading term in the Schwartz kernel of $f$ is the same as the one of $\tilde f$. Based on the expression of the kernel of $\tilde f$ given in \cite[Theorem 3.3]{LSSW16}, we have the following explicit expression for the kernel $K_f$.

\begin{lemma}\label{lm:Kf}
Let $f$ be a microlocally isotropic Gaussian random field of order $-m$ in $D\subset\mathbb R^d$. Denote by $H:=\frac{m-d}2$ the general Hurst parameter. 
\begin{itemize}
\item[(i)] If $H$ is a nonnegative integer, then
\[
K_f(x,y)=C_1(m,d)|x-y|^{2H}\ln|x-y|+r(x,y),
\]
where 
$C_1(m,d)=(-1)^{H+1}2^{-m+1}\pi^{-\frac d2}/(H!\Gamma(\frac m2))$ with $\Gamma(\cdot)$ being the Gamma function, and $r(x,y)$ denotes the residual which is more regular than the leading term.
\item[(ii)] If $H$ is not a nonnegative integer and $m>0$, then
\[
K_f(x,y)=C_2(m,d)|x-y|^{2H}+r(x,y),
\]
where  $C_2(m,d)=2^{-m}\pi^{-\frac
d2}\Gamma(-H)/\Gamma(\frac m2)$.
\item[(iii)] If $H$ is not a nonnegative integer and $m\in(-2n-2,-2n)$ with $n$ being a nonnegative integer, then
\[
K_f(x,y)=C_2(m,d)|x-y|^{2H}\left[1-\sum_{j=0}^n|x-y|^{2j}c_j\Delta^j\delta(x-y)\right]+r(x,y),
\]
where $c_0=1$ and 
\[
\quad c_j=\frac{A_d}{2^j j! d(d+2)\cdots(d+2j-2)}
\]
for $j\ge1$ with $A_d=2\pi^{\frac d2}/\Gamma(\frac d2)$ being the surface area of the unit sphere in $\mathbb R^d$, 
and $\delta(\cdot)$ is the Dirac delta function centered at $0$.
\item[(iv)] If $H$ is not a nonnegative integer and $m=-2n$ with $n$ being a nonnegative integer, then 
\[
K_f(x,y)=(-\Delta)^n\delta(x-y)+r(x,y).
\]
\end{itemize}
\end{lemma}

\begin{remark}\label{rk:Kf}
In cases (iii) and (iv) of Lemma \ref{lm:Kf}, all the partial derivatives for the Dirac delta function should be interpreted as distributions, and hence the kernels $K_f$ in these cases should also be interpreted as distributions (cf. \cite{LD72}). More precisely, for any test functions $\varphi,\psi\in\mathcal D$, $K_f$ given in (iii) and (iv) satisfies
\begin{align*}
&\int_{\mathbb R^d}\int_{\mathbb R^d}K_f(x,y)\varphi(x)\psi(y)dxdy\\
=&~C_2(m,d)\int_{\mathbb R^d}\int_{\mathbb R^d}|x-y|^{2H}\left[\varphi(x)\psi(y)-\sum_{j=0}^nc_j|x-y|^{2j}\varphi(x)\Delta^j\psi(x)\right]dxdy
\end{align*}
and
\begin{align*}
\int_{\mathbb R^d}\int_{\mathbb R^d}K_f(x,y)\varphi(x)\psi(y)dxdy
=\int_{\mathbb R^d}\varphi(x)(-\Delta)^n\psi(x)dx,
\end{align*}
respectively.
\end{remark}

\subsection{The fundamental solution}

Denote by $\Phi(x,y,k)$ the outgoing fundamental solution to the biharmonic wave operator $\mathcal L=\Delta^2-k^4$ such that
\begin{align}\label{eq:funda}
\Delta^2\Phi(x,y,k)-k^4\Phi(x,y,k)=-\delta(x-y)\quad\text{in} ~ \mathbb R^d,
\end{align}
where $\delta$ is the Dirac delta distribution. The expression of $\Phi$ can be obtained from two different approaches. 

The first approach makes use of the operator decomposition. Since the biharmonic wave operator can be written as the product of the Helmholtz and modified Helmholtz operators, i.e., $\mathcal L=(\Delta-k^2)(\Delta+k^2)$, the fundamental solution $\Phi$ is a linear composition of the fundamental solutions to the Helmholtz equation $(\Delta +k^2)u=0$ and the modified Helmholtz equation $(\Delta -k^2)u=0$, respectively. Hence, we may obtain that $\Phi$ depends on $|x-y|$ and is given in the form (cf. \cite{TH17,TS18})
\begin{align*}
\Phi(x,y,k)&=\frac{\rm i}{8k^2}\left(\frac{k}{2\pi|x-y|}\right)^{\frac{d-2}2}\left(H_{\frac{d-2}2}^{(1)}(k|x-y|)+\frac{2\rm i}{\pi}K_{\frac{d-2}2}(k|x-y|)\right)\\
&=\frac{\rm i}{8k^2}\left(\frac{k}{2\pi|x-y|}\right)^{\frac{d-2}2}\left(H_{\frac{d-2}2}^{(1)}(k|x-y|)+{\rm i}^{\frac{d}2+1}H_{\frac{d-2}2}^{(1)}({\rm i}k|x-y|)\right),
\end{align*}
where $H_{\nu}^{(1)}$
is the Hankel function of the first kind and order $\nu\in\mathbb R$, and 
\begin{align}\label{eq:Mac}
K_{\nu}(z)=\frac{\pi}2{\rm i}^{\nu+1}H_{\nu}^{(1)}({\rm i}z),\quad -\pi<\arg z\le\frac{\pi}2
\end{align}
is the Macdonald function (also known as the modified Bessel function of the second kind) of order $\nu\in\mathbb R$.
More precisely, we have 
\begin{equation}\label{eq:Phi}
\Phi(x,y,k)=\left\{
\begin{aligned}
&\frac{\rm i}{8k^2}\left(H_0^{(1)}(k|x-y|)-H_0^{(1)}({\rm i}k|x-y|)\right),\quad& d=2,\\
&\frac1{8\pi k^2|x-y|}\left(e^{{\rm i}k|x-y|}-e^{-k|x-y|}\right),\quad& d=3,
\end{aligned}
\right.
\end{equation}
where we use the fact
\[
H_{\frac12}^{(1)}(z)=\sqrt{\frac2{\pi z}}\frac{e^{{\rm i}z}}{\rm i}.
\]

The fundamental solution $\Phi$ may also be derived from the Fourier transform. Let  
\begin{align}\label{eq:Phik}
\Phi_k(x):=\mathcal F^{-1}\left[\frac1{|\xi|^4-k^4}\right](x),
\end{align}
where $\mathcal F^{-1}$ denotes the inverse Fourier transform. Taking the Fourier transform of \eqref{eq:funda} gives that $\Phi_k(x-y)$ also satisfies \eqref{eq:funda} and hence
\[
\Phi_k(x-y)=\Phi(x,y,k).
\]

\section{The direct problem}
\label{sec:direct}

In this section, we examine the well-posedness of the direct problem \eqref{eq:model}--\eqref{eq:radiation} in a proper sense when the source $f$ is a rough random field satisfying Assumption \ref{as:f}. The basic idea is to derive an equivalent  integral equation, which will also be used in the recovery of the strength for the random source. 

Using the fundamental solution $\Phi$ or $\Phi_k$ given in \eqref{eq:Phi} or \eqref{eq:Phik}, we define the volume potential
\begin{align*}
\mathcal H_k(\phi)(x):=-\int_{\mathbb R^d}\Phi(x,y,k)\phi(y)dy
=-(\Phi_k*\phi)(x),
\end{align*}
where $*$ denotes the convolution of $\Phi_k$ and $\phi$. 

\begin{lemma}\label{lm:Hk}
Let $B$ and $G$ be two bounded domains in $\mathbb R^d$. The operator $\mathcal H_k: H^{-s_1}(B)\to H^{s_2}(G)$ is bounded and satisfies
\[
\|\mathcal H_k\|_{\mathcal{L}(H^{-s_1}(B),H^{s_2}(G))}\lesssim\frac1{k^{3-s}}
\]
for $s:=s_1+s_2\in(0,3)$ with $s_1,s_2\ge0$.
\end{lemma}

\begin{proof}
For any $\phi\in C_0^\infty(B)$ and $\psi\in C_0^\infty(G)$, we still denote by $\phi$ and $\psi$ the zero extensions to $\mathbb R^d\setminus\overline B$ and $\mathbb R^d\setminus\overline G$, respectively. Then
\begin{align*}
\langle\mathcal H_k\phi,\psi\rangle&=\langle\widehat{\mathcal H_k\phi},\hat\psi\rangle=-\int_{\mathbb R^d}\frac1{|\xi|^4-k^4}\hat\phi(\xi)\hat\psi(\xi)d\xi\\
&=-\int_{\Omega_1}\frac{(1+|\xi|^2)^{\frac{s}2}}{|\xi|^4-k^4}\widehat{\mathcal J^{-s_1}\phi}(\xi)\widehat{\mathcal J^{-s_2}\psi}(\xi)d\xi-\int_{\Omega_2}\frac{(1+|\xi|^2)^{\frac{s}2}}{|\xi|^4-k^4}\widehat{\mathcal J^{-s_1}\phi}(\xi)\widehat{\mathcal J^{-s_2}\psi}(\xi)d\xi\\
&=:{\mathscr A}+{\mathscr B},
\end{align*}
where $\hat\phi=\mathcal{F}[\phi]$ is the Fourier transform of $\phi$,
\begin{align*}
\Omega_1:=&\left\{\xi\in\mathbb R^d:||\xi|-k|>\frac k2\right\}=\left\{\xi\in\mathbb R^d:|\xi|>\frac{3k}2\text{ or }|\xi|<\frac k2\right\},\\
\Omega_2:=&\left\{\xi\in\mathbb R^d:||\xi|-k|<\frac k2\right\}=\left\{\xi\in\mathbb R^d:\frac k2<|\xi|<\frac{3k}2\right\},
\end{align*}
and $\mathcal J^s:\mathcal S(\mathbb R^d)\to\mathcal S(\mathbb R^d)$ is the Bessel potential of order $s\in\mathbb R$ defined by (cf. \cite{LW21b})
\[
\mathcal J^s\phi:=(I-\Delta)^{\frac s2}\phi=\mathcal F^{-1}\left[(1+|\cdot|^2)^{\frac s2}\hat\phi\right]\quad\forall~\phi\in\mathcal S(\mathbb R^d)
\]
with $\mathcal S(\mathbb R^d)$ being the Schwartz space of all rapidly decreasing smooth functions.

For any $s\in(0,\frac32)$, the term $\mathscr A$ satisfies
\begin{align*}
|\mathscr A| &\le \int_{\Omega_1}\frac{(1+|\xi|^2)^{\frac{s}2}}{||\xi|-k|(|\xi|+k)(|\xi|^2+k^2)}|\widehat{\mathcal J^{-s_1}\phi}||\widehat{\mathcal J^{-s_2}\psi}|d\xi\\
&\lesssim\frac1k\int_{\{|\xi|>\frac{3k}2\}\cup\{|\xi|<\frac{k}2\}}\frac{(1+|\xi|^2)^{\frac{s}2}}{(|\xi|+k)(|\xi|^2+k^2)}|\widehat{\mathcal J^{-s_1}\phi}||\widehat{\mathcal J^{-s_2}\psi}|d\xi\\
&\lesssim \frac1k\int_{\{|\xi|>\frac{3k}2\}}\frac1{|\xi|^{3-s}}|\widehat{\mathcal J^{-s_1}\phi}||\widehat{\mathcal J^{-s_2}\psi}|d\xi+\frac1k\int_{\{|\xi|<\frac k2\}}\frac1{k^{3-s}}|\widehat{\mathcal J^{-s_1}\phi}||\widehat{\mathcal J^{-s_2}\psi}|d\xi\\
&\lesssim\frac1{k^{4-s}}\|\phi\|_{H^{-s_1}(B)}\|\psi\|_{H^{-s_2}(G)}.
\end{align*}

To estimate term $\mathscr B$, we use the change of variables
\[
\xi^*=\left(\frac{2k}{|\xi|}-1\right)\xi,
\]
which maps the domain $\Omega_{21}:=\{\xi:\frac k2<|\xi|<k\}$ to the domain $\Omega_{22}:=\{\xi:k<|\xi|<\frac{3k}2\}$, and has the Jacobian
\[
J(\xi)=\left|\det\left(\frac{\partial\xi^*}{\partial\xi}\right)\right|=\left(\frac{2k}{|\xi|}-1\right)^{d-1}.
\]
Then term $\mathscr B$ satisfies
\begin{align*}
\mathscr B&=-\int_{\Omega_{21}\cup\Omega_{22}}\frac{(1+|\xi|^2)^{\frac{s}2}}{|\xi|^4-k^4}\widehat{\mathcal J^{-s_1}\phi}(\xi)\widehat{\mathcal J^{-s_2}\psi}(\xi)d\xi\\
&=-\int_{\Omega_{21}}\frac{(1+|\xi|^2)^{\frac{s}2}}{|\xi|^4-k^4}\widehat{\mathcal J^{-s_1}\phi}(\xi)\widehat{\mathcal J^{-s_2}\psi}(\xi)d\xi-\int_{\Omega_{21}}\frac{(1+|\xi^*|^2)^{\frac{s}2}}{|\xi^*|^4-k^4}\widehat{\mathcal J^{-s_1}\phi}(\xi^*)\widehat{\mathcal J^{-s_2}\psi}(\xi^*)J(\xi)d\xi\\
&=-\int_{\Omega_{21}}\left[\frac{1}{|\xi|^4-k^4}+\frac{J(\xi)}{|\xi^*|^4-k^4}\right](1+|\xi|^2)^{\frac{s}2}\widehat{\mathcal J^{-s_1}\phi}(\xi)\widehat{\mathcal J^{-s_2}\psi}(\xi)d\xi\\
&\quad -\int_{\Omega_{21}}\frac{J(\xi)}{|\xi^*|^4-k^4}\left[(1+|\xi^*|^2)^{\frac{s}2}\widehat{\mathcal J^{-s_1}\phi}(\xi^*)\widehat{\mathcal J^{-s_2}\psi}(\xi^*)-(1+|\xi|^2)^{\frac{s}2}\widehat{\mathcal J^{-s_1}\phi}(\xi)\widehat{\mathcal J^{-s_2}\psi}(\xi)\right]d\xi\\
&=-\int_{\Omega_{21}}\left[\frac{1}{|\xi|^4-k^4}+\frac{J(\xi)}{|\xi^*|^4-k^4}\right](1+|\xi|^2)^{\frac{s}2}\widehat{\mathcal J^{-s_1}\phi}(\xi)\widehat{\mathcal J^{-s_2}\psi}(\xi)d\xi\\
&\quad -\int_{\Omega_{21}}\frac{J(\xi)}{|\xi^*|^4-k^4}\left[(1+|\xi^*|^2)^{\frac{s}2}-(1+|\xi|^2)^{\frac{s}2}\right]\widehat{\mathcal J^{-s_1}\phi}(\xi)\widehat{\mathcal J^{-s_2}\psi}(\xi)d\xi\\
&\quad -\int_{\Omega_{21}}\frac{J(\xi)}{|\xi^*|^4-k^4}(1+|\xi^*|^2)^{\frac{s}2}\left[\widehat{\mathcal J^{-s_1}\phi}(\xi^*)-\widehat{\mathcal J^{-s_1}\phi}(\xi)\right]\widehat{\mathcal J^{-s_2}\psi}(\xi)d\xi\\
&\quad -\int_{\Omega_{21}}\frac{J(\xi)}{|\xi^*|^4-k^4}(1+|\xi^*|^2)^{\frac{s}2}\widehat{\mathcal J^{-s_1}\phi}(\xi^*)\left[\widehat{\mathcal J^{-s_2}\psi}(\xi^*)-\widehat{\mathcal J^{-s_2}\psi}(\xi)\right]d\xi\\
&=:\mathscr B_1+\mathscr B_2+\mathscr B_3+\mathscr B_4.
\end{align*}
Note that
\begin{align*}
\left|\frac{1}{|\xi|^4-k^4}+\frac{J(\xi)}{|\xi^*|^4-k^4}\right|=&\left|\frac1{(|\xi|-k)(|\xi|+k)(|\xi|^2+k^2)}+\frac{2k-|\xi|}{|\xi|(k-|\xi|)(3k-|\xi|)(|\xi|^2-4k|\xi|+5k^2)}\right|\\
=&\left|\frac{2k(3|\xi|^2-6k|\xi|+k^2)}{|\xi|(|\xi|+k)(|\xi|^2+k^2)(3k-|\xi|)(|\xi|^2-4k|\xi|+5k^2)}\right|
\lesssim\frac1{k^4}
\end{align*}
 if $d=2$, and 
 \begin{align*}
 \left|\frac{1}{|\xi|^4-k^4}+\frac{J(\xi)}{|\xi^*|^4-k^4}\right|=&\left|\frac1{(|\xi|-k)(|\xi|+k)(|\xi|^2+k^2)}+\frac{(2k-|\xi|)^2}{|\xi|^2(k-|\xi|)(3k-|\xi|)(|\xi|^2-4k|\xi|+5k^2)}\right|\\
=&\left|\frac{-2(|\xi|^4-4k|\xi|^3+5k^2|\xi|^2-2k^3|\xi|-2k^4)}{|\xi|^2(|\xi|+k)(|\xi|^2+k^2)(3k-|\xi|)(|\xi|^2-4k|\xi|+5k^2)}\right|
\lesssim\frac1{k^4}
\end{align*}
if $d=3$,
which leads to
\[
|\mathscr B_1|\lesssim\frac1{k^{4-s}}\|\phi\|_{H^{-s_1}(B)}\|\psi\|_{H^{-s_2}(G)}.
\]
For term $\mathscr B_2$, since
\begin{align*}
&\left|\frac{J(\xi)}{|\xi^*|^4-k^4}\left[(1+|\xi^*|^2)^{\frac{s}2}-(1+|\xi|^2)^{\frac{s}2}\right]\right|\\
=&\left|\frac{(2k-|\xi|)^{d-1}(|\xi^*|^2-|\xi|^2)}{|\xi|^{d-1}(k-|\xi|)(3k-|\xi|)(|\xi|^2-4k|\xi|+5k^2)}s(1+\theta|\xi^*|^2+(1-\theta)|\xi|^2)^{s-1}\right|\\
=&\left|\frac{2(2k-|\xi|)^{d-1}(|\xi^*|+|\xi|)}{|\xi|^{d-1}(3k-|\xi|)(|\xi|^2-4k|\xi|+5k^2)}s(1+\theta|\xi^*|^2+(1-\theta)|\xi|^2)^{s-1}\right|
\lesssim\frac1{k^{4-s}}
\end{align*}
for some $\theta\in(0,1)$,
we then get
\[
|\mathscr B_2|\lesssim\frac1{k^{4-s}}\|\phi\|_{H^{-s_1}(B)}\|\psi\|_{H^{-s_2}(G)}.
\]
For term $\mathscr B_3$, it holds (cf. \cite[Theorem 3.2]{LW21b})
\begin{align*}
|\mathscr B_3| &\le\int_{\Omega_{21}}\left|\frac{J(\xi)(1+|\xi^*|^2)^{\frac{s}2}(|\xi^*|-|\xi|)}{|\xi^*|^4-k^4}\right|\left|M(|\nabla\widehat{\mathcal J^{-s_1}\phi}|)(\xi^*)+M(|\nabla\widehat{\mathcal J^{-s_1}\phi}|)(\xi)\right|\left|\widehat{\mathcal J^{-s_2}\psi}(\xi)\right|d\xi\\
&=\int_{\Omega_{21}}\left|\frac{2(2k-|\xi|)^{d-1}(1+|\xi^*|^2)^{\frac{s}2}}{|\xi|^{d-1}(3k-|\xi|)(|\xi|^2-4k|\xi|+5k^2)}\right|\\\\
&\qquad\times\left|M(|\nabla\widehat{\mathcal J^{-s_1}\phi}|)(\xi^*)+M(|\nabla\widehat{\mathcal J^{-s_1}\phi}|)(\xi)\right|\left|\widehat{\mathcal J^{-s_2}\psi}(\xi)\right|d\xi\\
&\lesssim\frac1{k^{3-s}}\|\phi\|_{H^{-s_1}(B)}\|\psi\|_{H^{-s_2}(G)},
\end{align*}
where $M(f)$ is the Hardy--Littlewood maximal function of $f$.
The term $\mathscr B_4$ can be estimated similarly. 

Combining the above estimates, we conclude that
\[
|\langle\mathcal H_k\phi,\psi\rangle|\lesssim\frac1{k^{3-s}}\|\phi\|_{H^{-s_1}(B)}\|\psi\|_{H^{-s_2}(G)}\quad\forall~\phi\in C_0^\infty(B),~\psi\in C_0^\infty(G),
\]
where $s\in(0,3)$. The proof is completed by extending the above result to $\phi\in H^{-s_1}(B)$ and $\psi\in H^{-s_2}(G)$ according to the facts $C_0^\infty(B)$ is dense in $L^2(B)$ and $H^{-s_1}(B)=\overline{L^2(B)}^{\|\cdot\|_{H^{-s_1}(B)}}$.
\end{proof}

\begin{theorem}\label{tm:wellposed}
Let $f$ satisfy Assumption \ref{as:f}. Then the problem \eqref{eq:model}--\eqref{eq:radiation} admits a unique solution 
\begin{align}\label{eq:u}
u(x;k)=-\int_D\Phi(x,y,k)f(y)dy
\end{align}
in the distribution sense such that $u\in W_{loc}^{\gamma,q}(\mathbb R^d)$ almost surely for any $q>1$ and $0<\gamma<\min\left\{\frac{6-d+m}2,\frac{6-d+m}2+\left(\frac1q-\frac12\right)d\right\}$.
\end{theorem}

\begin{proof}
The uniqueness can be proved similarly to the deterministic case given in \cite{LYZ}. It then suffices to show the existence and regularity of the solution. 

We first prove that the random field $u$ defined in \eqref{eq:u} is a solution of \eqref{eq:model}--\eqref{eq:radiation} in the distribution sense. In fact, for any test function $v\in\mathcal D$ with $\mathcal D$ being the $C_0^\infty(\mathbb R^d)$ equipped with a convex topology, it holds
\begin{align*}
\langle\Delta^2u-k^4u,v\rangle&=-\left\langle\int_{\mathbb R^d}(\Delta^2-k^4)\Phi(\cdot,y,k)f(y)dy,v\right\rangle\\
&=\left\langle\int_{\mathbb R^d}\delta(\cdot-y)f(y)dy,v\right\rangle=\langle f,v\rangle.
\end{align*}
Hence, $u=\mathcal H_kf$ satisfies \eqref{eq:model} in the distribution sense, where $f\in W^{\frac{m-d}2-\epsilon,p}(D)$ with $m\in(d-6,d]$ for any $\epsilon>0$ and $p>1$ according to Lemma \ref{lm:f}.
Moreover, for any $s_1\in(\frac{d-m}2,3)$ and $p\ge2$, the condition $\frac12>\frac1p-\frac{\frac{m-d}2-\epsilon+s_1}d$ is satisfied and hence the embedding
\[
W^{\frac{m-d}2-\epsilon,p}(D)\hookrightarrow H^{-s_1}(D)
\]
is continuous according to the Kondrachov embedding theorem. 

For any bounded domain $G\subset\mathbb R^d$ with a $C^1$-boundary, it follows from Lemma \ref{lm:Hk} that $\mathcal H_k: H^{-s_1}(D)\to H^{s_2}(G)$ is bounded for any positive $s_2<3-s_1<\frac{6-d+m}2$. 
Choosing $s_2=\frac{6-d+m}2-\epsilon$ for any sufficiently small $\epsilon>0$, 
then parameters $\gamma$ and $q$ given in the theorem satisfy $\gamma<s_2$ and $\frac1q>\frac12-\frac{s_2-\gamma}d$ such that the embedding
\[
\quad H^{s_2}(G)\hookrightarrow W^{\gamma,q}(G)
\]
is also continuous. We then conclude that $\mathcal H_k$ is bounded from $W^{\frac{m-d}2-\epsilon,p}(D)$ to $W^{\gamma,q}(G)$ with $p\ge2$, and hence $u=\mathcal H_kf\in W^{\gamma,q}(G)$, which completes the proof. 
\end{proof}

It is easy to verify that the solution $u=\mathcal H_kf$ obtained above is a linear combination of the solutions to the second order differential equations $\Delta u\pm k^2u=f$. In fact, we may rewrite the fundamental solution $\Phi$ as 
\begin{align*}
\Phi(x,y,k)=\frac1{2k^2}\Phi_+(x,y,k)-\frac1{2k^2}\Phi_-(x,y,k),
\end{align*}
where 
\begin{align*}
\Phi_+(x,y,k):=&\frac{\rm i}4\left(\frac{k}{2\pi|x-y|}\right)^{\frac{d-2}2}H_{\frac{d-2}2}^{(1)}(k|x-y|),\\
\Phi_-(x,y,k):=&\frac1{2\pi}\left(\frac{k}{2\pi|x-y|}\right)^{\frac{d-2}2}K_{\frac{d-2}2}(k|x-y|)
\end{align*}
are Green's functions to the second order linear operators $\Delta\pm k^2$ and satisfy 
\begin{align*}
\Delta \Phi_\pm(x,y,k)\pm k^2\Phi_\pm(x,y,k)=-\delta(x-y)\quad\text{in} ~ \mathbb R^d. 
\end{align*}
Then
\begin{align*}
v^\pm:=-\int_{\mathbb R^d}\Phi_\pm(x,y,k)\phi(y)dy
\end{align*}
are solutions of the equations
$\Delta v^{\pm}\pm k^2 v^{\pm}=f$
such that 
\begin{align}\label{eq:uv}
u=\frac1{2k^2}(v^+-v^-).
\end{align}

\section{The inverse problem}
\label{sec:inverse}

In this section, we study the inverse source problem, which is to determine the strength $\mu$ of the random source $f$ based on some proper data of the wave field $u$. Let $U\subset\mathbb R^d$ be the measurement domain, which is bounded and satisfies dist$(D,U)=r_0>0$. 

The inverse problem in the two dimensions is more tedious than the three-dimensional case due to the different form of the fundamental solution. For the two-dimensional problem, it requires an asymptotic expansion of the Hankel function and an additional truncation technique in order to get the recovery formula for $\mu$. In the following, we begin with the discussion on the three-dimensional case and then proceed to the more involved two-dimensional case.  

\subsection{The three-dimensional case} 

For the case $d=3$ and $m\in(-3,3]$, it follows from Assumption \ref{as:f} and \eqref{eq:Phi} that the distributional solution \eqref{eq:u} has the form
\begin{align}\label{eq:u3d}
u(x;k)=-\frac1{8\pi k^2}\int_{D}\frac{e^{{\rm i}k|x-y|}-e^{-k|x-y|}}{|x-y|}f(y)dy. 
\end{align} 

To get the recovery result based on the data from a single realization almost surely, the decay property of the solution with respect to the frequency is needed. 
According to the linear combination \eqref{eq:uv}, the required decay property of the solution $u$ can be obtained based on
an analogue of the ergodicity in the frequency domain of $v^+$ (cf. \cite{LHL20}) and the exponential decay property of $v^-$, which is stated in the following lemma.

\begin{lemma}\label{lm:u3d}
Let $f$ satisfy Assumption \ref{as:f} with $d=3$. For $k_1,k_2\ge1$, it holds uniformly for $x\in U$ that
\begin{align}\label{eq:ucouple1}
\big|\mathbb E\big[u(x;k_1)\overline{u(x;k_2)}\big]\big|&\lesssim k_1^{-2}k_2^{-2}\left[(k_1+k_2)^{-m}(1+|k_1-k_2|)^{-M_1}+k_1^{-M_2}+k_2^{-M_2}\right],\\\label{eq:ucouple2}
\left|\mathbb E\left[u(x;k_1)u(x;k_2)\right]\right|&\lesssim k_1^{-2}k_2^{-2}\left[(k_1+k_2)^{-M_1}(1+|k_1-k_2|)^{-m}+k_1^{-M_2}+k_2^{-M_2}\right],
\end{align}
where $M_1,M_2>0$ are arbitrary integers. 
In particular, if $k_1=k_2=k$, then
\begin{align}\label{eq:umoment}
\mathbb E|u(x;k)|^2=\left[\frac1{64\pi^2}\int_{D}\frac1{|x-\zeta|^2}\mu(\zeta)d\zeta\right] k^{-m-4}+O(k^{-m-5})
\end{align} 
as $k\to\infty$.
\end{lemma}

\begin{proof}
According to \eqref{eq:u3d}, we get
\begin{align*}
\mathbb E\big[u(x;k_1)\overline{u(x;k_2)}\big]&=\frac1{64\pi^2k_1^2k_2^2}\int_D\int_D\frac{e^{{\rm i}k_1|x-y|}-e^{-k_1|x-y|}}{|x-y|}\frac{e^{-{\rm i}k_2|x-z|}-e^{-k_2|x-z|}}{|x-z|}\mathbb E[f(y)f(z)]dydz\\
&=\frac1{64\pi^2k_1^2k_2^2}\int_D\int_D\frac{e^{{\rm i}(k_1|x-y|-k_2|x-z|)}}{|x-y||x-z|}K_f(y,z)dydz\\
&\quad -\frac1{64\pi^2k_1^2k_2^2}\int_D\int_D\frac{e^{{\rm i}k_1|x-y|-k_2|x-z|}+e^{-k_1|x-y|-{\rm i}k_2|x-z|}}{|x-y||x-z|}K_f(y,z)dydz\\
&\quad +\frac1{64\pi^2k_1^2k_2^2}\int_D\int_D\frac{e^{-k_1|x-y|-k_2|x-z|}}{|x-y||x-z|}K_f(y,z)dydz\\
&=:{\rm I}_1(x;k_1,k_2)+{\rm I}_2(x;k_1,k_2)+{\rm I}_3(x;k_1,k_2).
\end{align*}

The first term ${\rm I}_1$ has been estimated in \cite[Lemma A.1]{LLW} and satisfies
\begin{align}\label{eq:11}
|{\rm I}_1(x;k_1,k_2)|\lesssim k_1^{-2}k_2^{-2}(k_1+k_2)^{-m}(1+|k_1-k_2|)^{-M_1}
\end{align}
and 
\begin{align}\label{eq:31}
{\rm I}_1(x;k,k)&=\frac1{64\pi^2k^4}\left[\left(\int_{D}\frac1{|x-\zeta|}\mu(\zeta)d\zeta\right) k^{-m}+O(k^{-m-1})\right]\notag\\
&=\left[\frac1{64\pi^2}\int_{D}\frac1{|x-\zeta|^2}\mu(\zeta)d\zeta\right] k^{-m-4}+O(k^{-m-5}),
\end{align}
where $M_1>0$ is an arbitrary integer. 

The other two terms can be estimated by utilizing Lemma \ref{lm:Kf} and the exponential decay property of the integrant, i.e.,
$
e^{-k_1|x-y|}\le k_1^{-M_2}
$
for any $M_2>0$ since $|x-y|$ is bounded below and above for any $x\in U$ and $y\in D$. Without loss of generality, we only consider the leading term in the kernel function $K_f$ and omit the residual $r$ since it is more regular than the corresponding leading term. For $d=3$, we get $m\in(-3,3]$ according to Assumption \ref{as:f}. 
We take the term ${\rm I}_2$ as an example, whose estimate is given separately for different cases of $m$.

(i) The case $m\in(0,3]$. By Lemma \ref{lm:Kf}, it holds
\begin{equation*}
K_f(y,z)=\begin{cases}
C_1(m,3)\ln|y-z|,&\quad m=3,\\
C_2(m,3)|y-z|^{m-3},&\quad m\in(0,3),
\end{cases}
\end{equation*}

and hence
\[
\int_D\int_D|K_f(y,z)|dydz<\infty
\] 
due to the boundedness of the domain $D$. Then the term ${\rm I}_2$ satisfies
\begin{align*}
|{\rm I}_2(x;k_1,k_2)|&\lesssim k_1^{-2}k_2^{-2}\int_D\int_D\frac{e^{-k_2|x-z|}+e^{-k_1|x-y|}}{|x-y||x-z|}|K_f(y,z)|dydz\notag\\
&\lesssim k_1^{-2}k_2^{-2}\left(k_1^{-M_2}+k_2^{-M_2}\right),
\end{align*}
where $M_2>0$ is an arbitrary integer.

(ii) The case $m=0$. We get from Lemma \ref{lm:Kf} (iv) with $n=0$ that
\[
K_f(y,z)=\delta(y-z),
\]
which leads to
\begin{align*}
|{\rm I}_2(x;k_1,k_2)|&=\left|\frac1{64\pi^2k_1^2k_2^2}\int_D\frac{e^{({\rm i}k_1-k_2)|x-y|}+e^{(-k_1-{\rm i}k_2)|x-y|}}{|x-y|^2}dy\right|\\
&\lesssim k_1^{-2}k_2^{-2}\left(k_1^{-M_2}+k_2^{-M_2}\right)
\end{align*}
with $M_2>0$ being an arbitrary integer.

(iii) The case $m\in(-2,0)$. Utilizing Lemma \ref{lm:Kf} (iii) with $n=0$ and Remark \ref{rk:Kf}, we get
\begin{align*}
|{\rm I}_2(x;k_1,k_2)|&\le\frac{|C_2(m,3)|}{64\pi^2k_1^2k_2^2}\left|\int_D\int_D\frac{e^{{\rm i}k_1|x-y|}}{|x-y|}\left(\frac{e^{-k_2|x-z|}}{|x-z|}-\frac{e^{-k_2|x-y|}}{|x-y|}\right)|y-z|^{m-3}dydz\right|\\
&\quad +\frac{|C_2(m,3)|}{64\pi^2k_1^2k_2^2}\left|\int_D\int_D\frac{e^{-k_1|x-y|}}{|x-y|}\left(\frac{e^{-{\rm i}k_2|x-z|}}{|x-z|}-\frac{e^{-{\rm i}k_2|x-y|}}{|x-y|}\right)|y-z|^{m-3}dydz\right|.
\end{align*}
Since the estimates are the same for the two terms on the right-hand side of the above inequality, to estimate the term $\rm I_2$, it suffices to estimate the integral
\begin{align*}
\mathcal I(x,y):&=\int_D\left(\frac{e^{-k_2|x-z|}}{|x-z|}-\frac{e^{-k_2|x-y|}}{|x-y|}\right)|y-z|^{m-3}dz\\
&=\int_{D-\{y\}}\left(F_x(y+\tilde z)-F_x(y)\right)|\tilde z|^{m-3}d\tilde z
\end{align*}
for $x\in U$ and $y\in D$,
where 
\[
F_x(z):=\frac{e^{-k_2|x-z|}}{|x-z|}.
\]
It is clear to note that $F_x$ is smooth in $D$ and its derivatives decay exponentially. Define
\[
\tilde F_x(y,r)=\frac1{A_3}\int_{|\tilde z|=r}F_x(y+\tilde z)ds(\tilde z),
\]
where $A_3$ is the surface area of the unit sphere in $\mathbb R^3$ given in Lemma \ref{lm:Kf} and $R^*:=\max\limits_{y,z\in D}|y-z|$. We get from \cite[$(1.1.5)$]{LD72} that 
\begin{align*}
|\mathcal I(x,y)|&=\left|\int_{D-\{y\}}\left(F_x(y+\tilde z)-F_x(y)\right)|\tilde z|^{m-3}d\tilde z\right|\\
&=\left|A_3\int_0^{R^*}\left(\tilde F_x(y,r)-F_x(y)\right)r^{m-1}dr\right|
\lesssim k_2^{-M_2},
\end{align*}
where we use the fact
\[
|\tilde F_x(y,r)-F_x(y)|\lesssim k_2^{-M_2}r^2
\]
based on the Pizzetti formula (cf. \cite{LD72})
\[
\tilde F_x(y,r)=F_x(y)+\frac{\Delta F_x(y)}{2\cdot 1! d}r^2+\cdots+\frac{\Delta^jF_x(y)}{2^jj!d(d+2)\cdots(d+2j-2)}r^{2j}+\cdots\quad\text{as}\quad r\to0
\] 
and the exponential decay property of $F_x$.

(iii) The case $m=-2$. Based on Lemma \ref{lm:Kf} (iv) with $n=1$ and Remark \ref{rk:Kf}, it holds that 
\begin{align*}
|{\rm II}(x;k_1,k_2)|&=\frac1{64\pi^2k_1^2k_2^2}\left|\int_D\left(\frac{e^{{\rm i}k_1|x-y|}}{|x-y|}(-\Delta_y)\frac{e^{-k_2|x-y|}}{|x-y|}+\frac{e^{-k_1|x-y|}}{|x-y|}(-\Delta_y)\frac{e^{-{\rm i}k_2|x-y|}}{|x-y|}\right)dy\right|\\
&\lesssim k_1^{-2}k_2^{-2}\left(k_1^{-M_2}+k_2^{-M_2}\right),
\end{align*}
where we use again the smoothness and exponential decay property of the function $\frac{e^{-k_2|x-y|}}{|x-y|}$ for $x\in U$ and $y\in D$.

(iv) The case $m\in(-3,-2)$. This case can be proved through the same procedure used in the case (iii) by applying Lemma \ref{lm:Kf} (iii) with $n=1$ and the Pizzetti formula.

We can now conclude from the above discussions that
\begin{align}\label{eq:12}
|{\rm I}_2(x;k_1,k_2)|
\lesssim k_1^{-2}k_2^{-2}\left(k_1^{-M_2}+k_2^{-M_2}\right),
\end{align}
which also leads to
\begin{align}\label{eq:32}
{\rm I}_2(x;k,k)=O(k^{-m-5})
\end{align}
by choosing $M_2>m+1$.

Following the similar estimates as those for the term ${\rm I}_2$, we may show that the term ${\rm I}_3$ satisfies
\begin{align}\label{eq:13}
|{\rm I}_3(x;k_1,k_2)|\lesssim k_1^{-2-M_2}k_2^{-2-M_2}\lesssim k_1^{-2}k_2^{-2}\left(k_1^{-M_2}+k_2^{-M_2}\right)
\end{align}
and
\begin{align}\label{eq:33}
{\rm I}_3(x;k,k)=O(k^{-m-5}).
\end{align}

As a result, the estimate \eqref{eq:ucouple1} is proved by combining \eqref{eq:11}, \eqref{eq:12} and \eqref{eq:13}, and the estimate \eqref{eq:umoment} is concluded by using \eqref{eq:31}, \eqref{eq:32} and \eqref{eq:33}.
The proof is completed by noting that the formula \eqref{eq:ucouple2} can be estimated based on the same procedure as the proof of \eqref{eq:ucouple1}.
\end{proof}

\begin{theorem}\label{tm:u3d}
Let $f$ satisfy Assumption \ref{as:f} with $d=3$. For any $x\in U$, it holds almost surely that
\begin{align}\label{eq:u3dlimit}
\lim_{T\to\infty}\frac1T\int_T^{2T}k^{m+4}|u(x;k)|^2dk=\frac1{64\pi^2}\int_{D}\frac1{|x-\zeta|^2}\mu(\zeta)d\zeta=:T_3(x).
\end{align}
Moreover, the strength $\mu$ can be uniquely recovered by the measurement $\{T_3(x)\}_{x\in U}$.
\end{theorem}

\begin{proof}
If $T_3(x)$ is known for $x\in U$, which is smooth in $U$, then
the strength $\mu$ can be uniquely recovered by solving a deconvolution problem (cf. \cite[Theorem 1]{LPS08}).

Next, we prove \eqref{eq:u3dlimit}. It follows from Lemma \ref{lm:u3d} that
\[
\lim_{T\to\infty}\frac1T\int_T^{2T}k^{m+4}\mathbb E|u(x;k)|^2dk=\frac1{64\pi^2}\int_{D}\frac1{|x-\zeta|^2}\mu(\zeta)d\zeta.
\]
It then suffices to show that
\begin{align}\label{eq:error}
\lim_{T\to\infty}\frac1T\int_T^{2T}Y(x;k)dk=0
\end{align}
almost surely with 
\begin{align}\label{eq:Y}
Y(x;k):&=k^{m+4}\left(|u(x;k)|^2-\mathbb E|u(x;k)|^2\right)\notag\\
&=k^{m+4}\left(u_{\rm r}(x;k)^2-\mathbb E[u_{\rm r}(x;k)^2]\right)+k^{m+4}\left(u_{\rm i}(x;k)^2-\mathbb E[u_{\rm i}(x;k)^2]\right)
\end{align}
being a real-valued random process. Here, $u_{\rm r}:=\Re[u]$ and $u_{\rm i}:=\Im[u]$ denote the real and imaginary parts of $u$, respectively. Note that
\begin{align*}
\mathbb E\left|\frac1T\int_T^{2T}Y(x;k)dk\right|^2=\frac1{T^2}\int_T^{2T}\int_T^{2T}\mathbb E[Y(x;k_1)Y(x;k_2)]dk_1dk_2.
\end{align*}
To show \eqref{eq:error}, we only need to show 
\begin{align}\label{eq:Ycouple}
\lim_{T\to\infty}\frac1{T^2}\int_T^{2T}\int_T^{2T}\mathbb E[Y(x;k_1)Y(x;k_2)]dk_1dk_2=0.
\end{align} 

According to \eqref{eq:Y}, we get
\begin{align*}
\mathbb E[Y(x;k_1)Y(x;k_2)]&=k_1^{m+4}k_2^{m+4}\mathbb E\left[\left(u_{\rm r}(x;k_1)^2-\mathbb E[u_{\rm r}(x;k_1)^2]\right)\left(u_{\rm r}(x;k_2)^2-\mathbb E[u_{\rm r}(x;k_2)^2]\right)\right]\\
&\quad +k_1^{m+4}k_2^{m+4}\mathbb E\left[\left(u_{\rm r}(x;k_1)^2-\mathbb E[u_{\rm r}(x;k_1)^2]\right)\left(u_{\rm i}(x;k_2)^2-\mathbb E[u_{\rm i}(x;k_2)^2]\right)\right]\\
&\quad +k_1^{m+4}k_2^{m+4}\mathbb E\left[\left(u_{\rm i}(x;k_1)^2-\mathbb E[u_{\rm i}(x;k_1)^2]\right)\left(u_{\rm r}(x;k_2)^2-\mathbb E[u_{\rm r}(x;k_2)^2]\right)\right]\\
&\quad +k_1^{m+4}k_2^{m+4}\mathbb E\left[\left(u_{\rm i}(x;k_1)^2-\mathbb E[u_{\rm i}(x;k_1)^2]\right)\left(u_{\rm i}(x;k_2)^2-\mathbb E[u_{\rm i}(x;k_2)^2]\right)\right]\\
&=:{\mathcal Y}_1+{\mathcal Y}_2+{\mathcal Y}_3+{\mathcal Y}_4.
\end{align*}

It is shown in \cite[Lemma 4.2]{CHL19} that for two real-valued random variables $X$ and $Z$ with $(X,Z)$ being a Gaussian random vector and $\mathbb E[X]=\mathbb E[Z]=0$, it holds
\[
\mathbb E[(X^2-\mathbb EX^2)(Z^2-\mathbb EZ^2)]=2(\mathbb E[XZ])^2.
\]
Note that, for any fixed $x\in U$ and $k>1$, $u_{\rm r}(x;k)$ and $u_{\rm i}(x;k)$ are both real-valued Gaussian random variables. Hence, we obtain 
\begin{align*}
{\mathcal Y}_1&=2k_1^{m+4}k_2^{m+4}(\mathbb E[u_{\rm r}(x;k_1)u_{\rm r}(x;k_2)])^2\\
&=\frac12k_1^{m+4}k_2^{m+4}\left(\Re\big[\mathbb E[u(x;k_1)u(x;k_2)]+\mathbb E[u(x;k_1)\overline{u(x;k_2)}]\big]\right)^2.
\end{align*}
The estimate of ${\mathcal Y}_1$ is given below for the two different cases: (i) $m>0$ and (ii) $m\le0$.
 
(i) For the case $m>0$, by choosing $M_1=m$ in Lemma \ref{lm:u3d}, we get
\begin{align}\label{eq:i}
{\mathcal Y}_1 &\lesssim k_1^{m+4}k_2^{m+4}\left(k_1^{-2}k_2^{-2}\left[(k_1+k_2)^{-m}(1+|k_1-k_2|)^{-m}+k_1^{-M_2}+k_2^{-M_2}\right]\right)^2\notag\\
&\lesssim (1+|k_1-k_2|)^{-2m}+k_1^{-2M_2+m}k_2^m+k_1^mk_2^{-2M_2+m},
\end{align}
where we use the fact
\[
k_1^mk_2^m(k_1+k_2)^{-2m}=\left(\frac{k_1k_2}{(k_1+k_2)^2}\right)^m\le 1.
\]
Note that
\[
{\mathcal Y}_{11}:=\frac1{T^2}\int_T^{2T}\int_T^{2T}(1+|k_1-k_2|)^{-2m}dk_1dk_2
=\frac2{T^2}\int_T^{2T}\int_{k_2}^{2T}(1+k_1-k_2)^{-2m}dk_1dk_2.
\]
If $m=\frac12$,
\[
{\mathcal Y}_{11}=\frac2{T^2}\int_T^{2T}\ln(1+2T-k_2)dk_2\le\frac2T\ln(1+2T).
\]
If $m=1$,
\[
{\mathcal Y}_{11}=\frac2T-\frac2{T^2}\ln(1+T).
\]
If $m\neq\frac12,1$,
\[
{\mathcal Y}_{11}=
\frac{2-2(1+T)^{2-2m}}{T^2(2-2m)(1-2m)}-\frac2{(1-2m)T}.
\]
The above estimates lead to
\begin{align}\label{eq:i1}
\lim_{T\to\infty}\frac1{T^2}\int_T^{2T}\int_T^{2T}(1+|k_1-k_2|)^{-2m}dk_1dk_2=0
\end{align}
for $m>0$. 
Moreover, by choosing $M_2>m$, we have 
\begin{align}\label{eq:i2}
\lim_{T\to\infty}\frac1{T^2}\int_T^{2T}\int_T^{2T}k_1^{-2M_2+m}k_2^mdk_1dk_2&=\lim_{T\to\infty}\frac1{T^2}\frac{(2T)^{-2M_2+m+1}-T^{-2M_2+m+1}}{-2M_2+m+1}\frac{(2T)^{m+1}-T^{m+1}}{m+1}\notag\\
&=\lim_{T\to\infty}\frac{(2^{-2M_2+m+1}-1)(2^{m+1}-1)}{(-2M_2+m+1)(m+1)}T^{-2M_2+2m}=0,
\end{align}
which, together with \eqref{eq:i} and \eqref{eq:i1}, leads to
\[
\lim_{T\to\infty}\frac1{T^2}\int_T^{2T}\int_T^{2T}{\mathcal Y}_1dk_1dk_2=0.
\]

(ii) For the case $m\le0$, an application of Lemma \ref{lm:u3d} yields 
\begin{align}\label{eq:ii}
{\mathcal Y}_1 &\lesssim k_1^{m+4}k_2^{m+4}\left(k_1^{-2}k_2^{-2}\left[(k_1+k_2)^{-m}(1+|k_1-k_2|)^{-M_1}+k_1^{-M_2}+k_2^{-M_2}\right]\right)^2\notag\\
&\lesssim k_1^mk_2^m(k_1+k_2)^{-2m}(1+|k_1-k_2|)^{-2M_1}+k_1^{-2M_2+m}k_2^m+k_1^mk_2^{-2M_2+m}
\end{align}
due to the fact
\[
(k_1+k_2)^{-M_1}(1+|k_1-k_2|)^{-m}\lesssim(1+|k_1-k_2|)^{-M_1}(k_1+k_2)^{-m}.
\]
It is easy to obtain
\begin{align}\label{eq:ii1}
\lim_{T\to\infty}\frac1{T^2}\int_T^{2T}\int_T^{2T}\left(k_1^{-2M_2+m}k_2^m+k_1^mk_2^{-2M_2+m}\right)dk_1dk_2=0
\end{align}
for any $M_2>0$ according to \eqref{eq:i2}. In addition,
\begin{align*}
&\frac1{T^2}\int_T^{2T}\int_T^{2T}k_1^mk_2^m(k_1+k_2)^{-2m}(1+|k_1-k_2|)^{-2M_1}dk_1dk_2\\
&\lesssim\left(\frac1{T^2}\int_T^{2T}\int_T^{2T}k_1^{2m}k_2^{2m}(k_1+k_2)^{-4m}dk_1dk_2\right)^{\frac12}\left(\frac1{T^2}\int_T^{2T}\int_T^{2T}(1+|k_1-k_2|)^{-4M_1}dk_1dk_2\right)^{\frac12}\\
&\to ~0\quad\text{as}\quad T\to\infty
\end{align*}
for any $M_1>0$ based on \eqref{eq:i1} and the fact
\begin{align*}
\frac1{T^2}\int_T^{2T}\int_T^{2T}k_1^{2m}k_2^{2m}(k_1+k_2)^{-4m}dk_1dk_2\lesssim\frac1{T^2}\int_T^{2T}\int_T^{2T}\left(k_1^{-2m}k_2^{2m}+k_1^{2m}k_2^{-2m}\right)dk_1dk_2\lesssim1.
\end{align*}
The above estimate together with \eqref{eq:ii} and \eqref{eq:ii1} also gives rise to
\[
\lim_{T\to\infty}\frac1{T^2}\int_T^{2T}\int_T^{2T}{\mathcal Y}_1dk_1dk_2=0.
\]

The terms ${\mathcal Y}_2$, ${\mathcal Y}_3$ and ${\mathcal Y}_4$ can be estimated similarly. The details are omitted for brevity. Combining these estimates yields \eqref{eq:Ycouple} and completes the proof.
\end{proof}

\subsection{The two-dimensional case}

For the case $d=2$ and $m\in(-4,2]$, we obtain from Assumption \ref{as:f} and \eqref{eq:Phi} that the distributional solution given in \eqref{eq:u} takes the form
\begin{align}\label{eq:u2d}
u(x;k)=-\frac{\rm i}{8k^2}\int_D\left(H_0^{(1)}(k|x-y|)-H_0^{(1)}({\rm i}k|x-y|)\right)f(y)dy.
\end{align}

To get the recovery formula of the strength $\mu$ for the random source $f$, we recall the following asymptotic expansion of the Hankel function on $\mathbb C$ (cf. \cite{AS92}):
\[
H_0^{(1)}(z)=\sum_{j=0}^\infty a_jz^{-(j+\frac12)}e^{{\rm i}z}\quad \text{as}\quad|z|\to\infty,
\]
where 
\[
a_0=\sqrt{\frac2{\pi}}e^{-\frac{{\rm i}\pi}4},\quad 
a_j=\sqrt{\frac2{\pi}}\left(\frac{\rm i}8\right)^j\left(\prod_{l=1}^j(2l-1)^2/j!\right)e^{-\frac{{\rm i}\pi}4},\quad j\ge1.
\]
Define the truncated functions
\begin{align*}
H_{0,N}^{(1)}(z):&=\sum_{j=0}^Na_jz^{-(j+\frac12)}e^{{\rm i}z},\\
\Phi_N(x,y,k):&=\frac{\rm i}{8k^2}\left(H_{0,N}^{(1)}(k|x-y|)-H_{0,N}^{(1)}({\rm i}k|x-y|)\right).
\end{align*}
For any $N\in\mathbb N$, a simple calculation yields 
\begin{align}\label{eq:Phierror}
\Phi(x,y,k)-\Phi_N(x,y,k)&=\frac{\rm i}{8k^2}\sum_{j=N+1}^\infty\left(a_j(k|x-y|)^{-(j+\frac12)}e^{{\rm i}k|x-y|}-a_j({\rm i}k|x-y|)^{-(j+\frac12)}e^{-k|x-y|}\right)\notag\\
&=O\left(\frac1{k^2(k|x-y|)^{N+\frac32}}\right)\quad\text{as}\quad k|x-y|\to\infty.
\end{align}
Based on the truncated fundamental solution $\Phi_N$ with $N=3$, we consider the truncated solution 
\begin{align*}
u_3(x;k):&=-\int_D\Phi_{3}(x,y,k)f(y)dy\\
&=-\frac{\rm i}{8k^2}\int_D\left(H_{0,3}^{(1)}(k|x-y|)-H_{0,3}^{(1)}({\rm i}k|x-y|)\right)f(y)dy\\
&=-\frac{\rm i}{8k^2}\sum_{j=0}^3a_j\int_D\left((k|x-y|)^{-(j+\frac12)}e^{{\rm i}k|x-y|}-({\rm i}k|x-y|)^{-(j+\frac12)}e^{-k|x-y|}\right)f(y)dy.
\end{align*}

\begin{lemma}\label{lm:u2error}
Let $f$ satisfy Assumption \ref{as:f} with $d=2$. For $k\gg1$ and $x\in U$, the error between the solution $u$ and the truncated solution $u_3$ satisfies
\begin{equation*}
|u(x;k)-u_3(x;k)|\lesssim \begin{cases}
k^{-\frac72},\quad &m\in(-4,0],\\
k^{-\frac{11}2},\quad &m\in(0,2],
\end{cases}
\end{equation*}
almost surely.
\end{lemma}

\begin{proof}
According to \eqref{eq:Phierror}, for $y\in D$ and $x\in U$  with dist$(D,U)=r_0>0$, we get
\begin{align*}
|\Phi(x,y,k)-\Phi_3(x,y,k)|&=O\left(k^{-\frac{13}2}|x-y|^{-\frac92}\right),\\
|\partial_{y_i}\Phi(x,y,k)-\partial_{y_i}\Phi_3(x,y,k)|&=O\left(k^{-\frac{11}2}|x-y|^{-\frac92}\right),\\
|\partial^2_{y_iy_j}\Phi(x,y,k)-\partial^2_{y_iy_j}\Phi_3(x,y,k)|&=O\left(k^{-\frac92}|x-y|^{-\frac92}\right),\\
|\partial^3_{y_iy_jy_l}\Phi(x,y,k)-\partial^3_{y_iy_jy_l}\Phi_3(x,y,k)|&=O\left(k^{-\frac72}|x-y|^{-\frac92}\right).
\end{align*}

If $m\in(-4,0]$, then $f\in W^{\frac{m-d}2-\epsilon,p}(D)\subset W^{-3,p}(D)$ for any $p>1$ according to Lemma \ref{lm:f}. We then get
\[
|u(x;k)-u_3(x;k)|\le\|\Phi(x,\cdot,k)-\Phi_3(x,\cdot,k)\|_{W^{3,q}(D)}\|f\|_{W^{-3,p}(D)}\lesssim k^{-\frac72}
\]
with $q$ satisfying $\frac1p+\frac1q=1$.

If $m\in(0,2]$, then $f\in W^{\frac{m-d}2-\epsilon,p}(D)\subset W^{-1,p}(D)$ for any $p>1$, and hence
\[
|u(x;k)-u_3(x;k)|\le\|\Phi(x,\cdot,k)-\Phi_3(x,\cdot,k)\|_{W^{1,q}(D)}\|f\|_{W^{-1,p}(D)}\lesssim k^{-\frac{11}2},
\] 
which completes the proof.
\end{proof}

Similar to the three-dimensional case, to get the recovery formula in the almost surely sense, we need to show the asymptotical independence of the truncated solution $u_3$, which is stated in the following lemma. 

\begin{lemma}\label{lm:u2d}
Let $f$ satisfy Assumption \ref{as:f}. For $k_1,k_2\ge1$, it holds uniformly for $x\in U$ that
\begin{align}\label{eq:u2couple1}
\big|\mathbb E\big[u_3(x;k_1)\overline{u_3(x;k_2)}\big]\big|&\lesssim k_1^{-\frac52}k_2^{-\frac52}\left[(k_1+k_2)^{-m}(1+|k_1-k_2|)^{-M_1}+k_1^{-M_2}+k_2^{-M_2}\right],\\\label{eq:u2couple2}
\big|\mathbb E\big[u_3(x;k_1)u_3(x;k_2)\big]\big|&\lesssim k_1^{-\frac52}k_2^{-\frac52}\left[(k_1+k_2)^{-M_1}(1+|k_1-k_2|)^{-m}+k_1^{-M_2}+k_2^{-M_2}\right],
\end{align}
where $M_1,M_2>0$ are arbitrary integers. 
In particular, if $k_1=k_2=k$, then
\begin{align}\label{eq:u2moment}
\mathbb E|u_3(x;k)|^2=\left[\frac1{32\pi}\int_{D}\frac1{|x-\zeta|}\mu(\zeta)d\zeta\right] k^{-m-5}+O(k^{-m-6}).
\end{align} 

\end{lemma}

\begin{proof}
The truncated solution $u_3$ at two different frequencies $k_1$ and $k_2$ satisfies
\begin{align*}
&\mathbb E\big[u_3(x;k_1)\overline{u_3(x;k_2)}\big]\\
&=\frac1{64k_1^2k_2^2}\sum_{j,l=0}^3a_j\overline{a_l}\int_D\int_D\left((k_1|x-y|)^{-(j+\frac12)}e^{{\rm i}k_1|x-y|}-({\rm i}k_1|x-y|)^{-(j+\frac12)}e^{-k_1|x-y|}\right)\\
&\quad\times\left((k_2|x-z|)^{-(l+\frac12)}e^{-{\rm i}k_2|x-z|}-(-{\rm i}k_2|x-z|)^{-(l+\frac12)}e^{-k_2|x-z|}\right)\mathbb E[f(y)f(z)]dydz\\
&=\frac1{64k_1^2k_2^2}\sum_{j,l=0}^3\frac{a_j\overline{a_l}}{k_1^{j+\frac12}k_2^{l+\frac12}}\int_D\int_D\frac{e^{{\rm i}(k_1|x-y|-k_2|x-z|)}}{|x-y|^{j+\frac12}|x-z|^{l+\frac12}}\mathbb E[f(y)f(z)]dydz\\
&\quad -\frac1{64k_1^2k_2^2}\sum_{j,l=0}^3\frac{a_j\overline{a_l}}{k_1^{j+\frac12}(-{\rm i}k_2)^{l+\frac12}}\int_D\int_D\frac{e^{{\rm i}k_1|x-y|-k_2|x-z|}}{|x-y|^{j+\frac12}|x-z|^{l+\frac12}}\mathbb E[f(y)f(z)]dydz\\
&\quad -\frac1{64k_1^2k_2^2}\sum_{j,l=0}^3\frac{a_j\overline{a_l}}{({\rm i}k_1)^{j+\frac12}k_2^{l+\frac12}}\int_D\int_D\frac{e^{-k_1|x-y|-{\rm i}k_2|x-z|}}{|x-y|^{j+\frac12}|x-z|^{l+\frac12}}\mathbb E[f(y)f(z)]dydz\\
&\quad +\frac1{64k_1^2k_2^2}\sum_{j,l=0}^3\frac{a_j\overline{a_l}}{({\rm i}k_1)^{j+\frac12}(-{\rm i}k_2)^{l+\frac12}}\int_D\int_D\frac{e^{-k_1|x-y|-k_2|x-z|}}{|x-y|^{j+\frac12}|x-z|^{l+\frac12}}\mathbb E[f(y)f(z)]dydz\\
&=:{\rm J}_1(x;k_1,k_2)+{\rm J}_2(x;k_1,k_2)+{\rm J}_3(x;k_1,k_2)+{\rm J}_4(x;k_1,k_2).
\end{align*}

For the term ${\rm J}_1$, we have from \cite[Lemma A.1]{LLW} that 
\[
|{\rm J}_1(x;k_1,k_2)|\lesssim k_1^{-\frac52}k_2^{-\frac52}(k_1+k_2)^{-m}(1+|k_1-k_2|)^{-M_1}
\]
and 
\begin{align*}
{\rm J}_1(x;k,k)&=\frac{|a_0|^2}{64k^5}\int_D\int_D\frac{e^{{\rm i}(k|x-y|-k|x-z|)}}{|x-y|^{\frac12}|x-z|^{\frac12}}\mathbb E[f(y)f(z)]dydz\\
&\quad +\frac1{64k^4}\sum_{\substack{j,l=0\\j\,\text{or}\,l\neq0}}^3\frac{a_j\overline{a_l}}{k^{j+l+1}}\int_D\int_D\frac{e^{{\rm i}(k|x-y|-k|x-z|)}}{|x-y|^{j+\frac12}|x-z|^{l+\frac12}}\mathbb E[f(y)f(z)]dydz\\
&=\frac1{32\pi}\left[\int_D\frac1{|x-\zeta|}\mu(\zeta)d\zeta\right]k^{-m-5}+O(k^{-m-6}),
\end{align*}
where $M_1>0$ is an arbitrary integer. 

Similar to the three-dimensional case, the other three terms can be estimated by taking advantage of the exponential decay of the integrants. We then obtain 
\[
|{\rm J}_2(x;k_1,k_2)+{\rm J}_3(x;k_1,k_2)+{\rm J}_4(x;k_1,k_2)|\lesssim k_1^{-\frac52}k_2^{-\frac52}\left(k_1^{-M_2}+k_2^{-M_2}\right)
\]
for any $M_2>0$ and
\[
{\rm J}_2(x;k,k)+{\rm J}_3(x;k,k)+{\rm J}_4(x;k,k)=O(k^{-m-6})
\]
by choosing $M_2>m+1$.

The estimates above lead to \eqref{eq:u2couple1} and \eqref{eq:u2moment}. The proof of \eqref{eq:u2couple2} is to combine a similar proof of \eqref{eq:u2couple1} and \cite[Corollary 5.4]{LLW}.
\end{proof}

Based on the estimates for the truncated solution $u_3$, the  unique recovery of the strength can be obtained by a single realization of the wave field $u$ in the almost surely sense, which is stated in the following theorem.

\begin{theorem}\label{tm:u2d}
Let $f$ satisfy Assumption \ref{as:f}. For any $x\in U$, it holds almost surely that
\begin{align}\label{eq:u2dlimit}
\lim_{T\to\infty}\frac1T\int_T^{2T}k^{m+5}|u(x;k)|^2dk=\frac1{32\pi}\int_D\frac1{|x-\zeta|}\mu(\zeta)d\zeta=:T_2(x),
\end{align}
and the strength $\mu$ can be uniquely determined by the measurement $\{T_2(x)\}_{x\in U}$.
\end{theorem}

\begin{proof}
Using \eqref{eq:u2moment} in Lemma \ref{lm:u2d}, we get for $x\in U$ that 
\begin{align}\label{eq:u2ex}
\lim_{T\to\infty}\frac1T\int_T^{2T}k^{m+5}\mathbb E|u_3(x;k)|^2dk=\frac1{32\pi}\int_D\frac1{|x-\zeta|}\mu(\zeta)d\zeta.
\end{align}
First we show that 
\begin{align}\label{eq:u2as}
\lim_{T\to\infty}\frac1T\int_T^{2T}k^{m+5}|u_3(x;k)|^2dk=\frac1{32\pi}\int_D\frac1{|x-\zeta|}\mu(\zeta)d\zeta
\end{align}
in the almost surely sense. In fact, following the same procedure as the proof of \eqref{eq:error} in Theorem \ref{tm:u3d} and utilizing Lemma \ref{lm:u2d}, we have almost surely that
\[
\lim_{T\to\infty}\frac1T\int_T^{2T}k^{m+5}\left(|u_3(x;k)|^2-\mathbb E|u_3(x;k)|^2\right)dk=0,
\]
which, together with \eqref{eq:u2ex}, leads to \eqref{eq:u2as}.

Note that
\begin{align*}
\frac1T\int_T^{2T}k^{m+5}|u(x;k)|^2dk &=\frac1T\int_T^{2T}k^{m+5}|u_3(x;k)|^2dk\\
&\quad +\frac1T\int_T^{2T}k^{m+5}|u(x;k)-u_3(x;k)|^2dk\\
&\quad +\frac2T\int_T^{2T}k^{m+5}\Re\left[\overline{u_3(x;k)}(u(x;k)-u_3(x;k))\right]dk,
\end{align*}
where 
\begin{align*}
&\frac2T\int_T^{2T}k^{m+5}\Re\left[\overline{u_3(x;k)}(u(x;k)-u_3(x;k))\right]dk\\
&\lesssim\left[\frac1T\int_T^{2T}k^{m+5}|u_3(x;k)|^2dk\right]^{\frac12}\left[\frac1T\int_T^{2T}k^{m+5}|u(x;k)-u_3(x;k)|^2dk\right]^{\frac12}.
\end{align*}
As a result, to prove \eqref{eq:u2dlimit}, it suffices to show
\[
\lim_{T\to\infty}\frac1T\int_T^{2T}k^{m+5}|u(x;k)-u_3(x;k)|^2dk=0.
\]

For the case $m\in(-4,0]$, according to Lemma \ref{lm:u2error}, it holds
\begin{align*}
\frac1T\int_T^{2T}k^{m+5}|u(x;k)-u_3(x;k)|^2dk\lesssim\frac1T\int_T^{2T}k^{m-2}dk\to 0\quad\text{as}~T\to\infty.
\end{align*}

For the case $m\in(0,2]$, an application of Lemma \ref{lm:u2error} leads to
\begin{align*}
\frac1T\int_T^{2T}k^{m+5}|u(x;k)-u_3(x;k)|^2dk\lesssim\frac1T\int_T^{2T}k^{m-6}dk\to 0\quad\text{as}~T\to\infty,
\end{align*}
which completes the proof.
\end{proof}

\section{Numerical experiments}
\label{sec:num}

In this section, we present some numerical experiments to demonstrate the validity and effectiveness of the proposed method. 
Specifically, we consider the case $d=2$ and $m=0$, i.e., the random source is generated by the white noise in the form 
\[
f=\sqrt{\mu}\dot{W},
\]
where $\mu\in C_0^\infty(D)$ and $\mu\ge0$.

\subsection{The reconstruction formula}\label{sec:5.1}

When the random source is taken as the white noise model, both the covariance operator and its symbol have simpler forms. Based on the It\^o isometry, the covariance operator $\mathcal Q_f$ is given explicitly by
\[
\langle\mathcal Q_f\varphi,\psi\rangle=\mathbb E[\langle f,\varphi\rangle\langle f,\psi\rangle]=\mathbb E\left[\int_D\varphi(x)\sqrt{\mu(x)}dW(x)\int_D\psi(y)\sqrt{\mu(y)}dW(y)\right]=\langle\mu\varphi,\psi\rangle
\]
for any $\varphi,\psi\in\mathcal D$, which implies
\[
(\mathcal Q_f\varphi)(x)=\mu(x)\varphi(x)=\frac1{(2\pi)^2}\int_{\mathbb R^2}e^{{\rm i}x\cdot\xi}\mu(x)\hat\varphi(\xi)d\xi.
\]
Hence, the symbol $\sigma(x,\xi)$ of the pseudo-differential operator $\mathcal Q_f$ has only one term $\sigma(x,\xi)=\mu(x)$ with $\mu$ being the strength of the source $f$. 

Consequently, when using the second moment of $u$ to recover the strength $\mu$, the wave number $k$ is not required to be sufficiently large for the white noise case. More precisely, according to the expression of the solution given in \eqref{eq:u2d}, we get 
\begin{align}\label{eq:u_exa}
u(x;k)=-\frac{\rm i}{8k^2}\int_D\left(H_0^{(1)}(k|x-y|)-H_0^{(1)}({\rm i}k|x-y|)\right)\sqrt{\mu(y)}dW(y),
\end{align}
which leads to
\begin{align}\label{eq:data1}
64k^4\mathbb E|u(x;k)|^2=&\int_D\left|H_0^{(1)}(k|x-y|)-H_0^{(1)}({\rm i}k|x-y|)\right|^2\mu(y)dy.
\end{align}
Noting that the function $|H_0^{(1)}(k|x-y|)-H_0^{(1)}({\rm i}k|x-y|)|^2$ involved in the above integral has no singularity, we get theoretically that the strength $\mu$ can be uniquely determined through \eqref{eq:data1} at a single frequency.
However, similar to the inverse random source problem for elastic waves studied in \cite{BCL17},  the numerical solution is rather unstable if one uses the numerical integration of \eqref{eq:data1} directly to recover the strength $\mu$ due to the fast decay of its singular values. To handle the instability, a modified integral equation and regularization technique are required to get a more stable and accurate result.

Note that $H_0^{(1)}=J_0+{\rm i}Y_0$ with $J_0$ and $Y_0$ being the real-valued Bessel functions of the first kind and the second kind, respectively, and the function
\[
{\rm i}H_0^{(1)}({\rm i}k|x-y|)=\frac2{\pi}K_0(k|x-y|)
\]
obtained by \eqref{eq:Mac} with $d=2$ is also real-valued.
We then split the solution $u$ into its real and imaginary parts as follows:
\begin{align*}
\Re[u(x;k)]&=\frac1{8k^2}\int_D\left(Y_0(k|x-y|)+{\rm i}H_0^{(1)}({\rm i}k|x-y|)\right)\sqrt{\mu(y)}dW(y),\\
\Im[u(x;k)]&=-\frac1{8k^2}\int_DJ_0(k|x-y|)\sqrt{\mu(y)}dW(y),
\end{align*}
and use the modified integral equation
\begin{align}\label{eq:data2}
&64k^4\mathbb E\left[(\Re[u(x;k)])^2-(\Im[u(x;k)])^2\right]\notag\\
&=\int_D\left[\left(Y_0(k|x-y|)+{\rm i}H_0^{(1)}({\rm i}k|x-y|)\right)^2-\left(J_0(k|x-y|)\right)^2\right]\mu(y)dy\notag\\
&=:\int_DG(x-y)\mu(y)dy
\end{align}
to reconstruct the strength $\mu$.

\subsection{The synthetic data}

The direct problem is solved numerically to generate the synthetic data. In the experiments, we choose a square domain $D:=[-1,1]\times[-1,1]$ for the support and the measurement domain $U$, which is specified in the next subsection, such that dist$(D,U)>0$.
For square domains $D$ and $U$, we define two index sets
\begin{align*}
\mathcal T_U:&=\{{\boldsymbol i}=(i_1,i_2): i_l=0,\cdots,N_U,~l=1,2\},\\
\mathcal T_D:&=\{{\boldsymbol j}=(j_1,j_2): j_l=0,\cdots,N_D,~l=1,2\}
\end{align*}
with $N_U=40$ and $N_D=20$, and define two sets of discrete points 
\begin{align*}
\{x_{\boldsymbol i}\}_{\boldsymbol i\in\mathcal T_U}:&=\left\{x_{\boldsymbol i}=(x^{(1)}_{i_1},x^{(2)}_{i_2})^\top\in U: x_{\boldsymbol i}=x_{(0,0)}+(i_1\delta x,i_2\delta x)^\top\right\}_{\boldsymbol i\in\mathcal T_U},\\
\{y_{\boldsymbol j}\}_{\boldsymbol j\in\mathcal T_D}:&=\left\{y_{\boldsymbol j}=(y^{(1)}_{j_1},y^{(2)}_{j_2})^\top\in D: y_{\boldsymbol j}=y_{(0,0)}+(j_1\delta y,j_2\delta y)^\top\right\}_{\boldsymbol i\in\mathcal T_D},
\end{align*}
where $\delta x=1/N_U$ and $\delta y=1/N_D$.  The synthetic data is generated at the discrete points $\{x_{\boldsymbol i}\}_{\boldsymbol i\in\mathcal T_U}$, and the solution $u(x_{\boldsymbol i};k)$ is approximated  through the numerical quadrature of the It\^o integral \eqref{eq:u_exa} by
\[
u(x_{\boldsymbol i};k)\approx u_{\rm num}(x_{\boldsymbol i},\omega,k):=\frac1{8{\rm i}k^2}\sum_{\boldsymbol j\in\mathcal T_D}\left(H_0^{(1)}(k|x_{\boldsymbol i}-y_{\boldsymbol j}|)-H_0^{(1)}({\rm i}k|x_{\boldsymbol i}-y_{\boldsymbol j}|)\right)\sqrt{\mu(y_{\boldsymbol j})}\delta_{\boldsymbol j} W,
\]
where
\begin{align*}
\delta_{\boldsymbol j} W:=\int_{I_{\boldsymbol j}}dW(y)\overset{d}{=}\sqrt{|I_{\boldsymbol j}|}\xi_{\boldsymbol j}.
\end{align*}
Here, the notation $A\overset{d}{=}B$ means that $A$ and $B$ have the same distribution, $\{\xi_{\boldsymbol j}\}_{\boldsymbol j\in\mathcal T_D}$ is a set of independent identically distributed normal random variables, $I_{\boldsymbol j}=[j_1\delta y,(j_1+1)\delta y]\times[j_2\delta y,(j_2+1)\delta y]$ is a square with side length $\delta y$, and $|I_{\boldsymbol j}|$ is the area of $I_{\boldsymbol j}$.

\subsection{The numerical method}

According to \eqref{eq:data2}, we define the measurement
\[
\mathcal M(x,k)=64k^4\mathbb E\left[(\Re[u(x;k)])^2-(\Im[u(x;k)])^2\right],\quad x\in U.
\]
Then its evaluation at the discrete points $\{x_{\boldsymbol i}\}_{\boldsymbol i\in\mathcal T_U}$ can be approximated by
\[
\mathcal M(x_{\boldsymbol i},k)\approx\sum_{\boldsymbol j\in\mathcal T_D}|I_{\boldsymbol j}|G(x_{\boldsymbol i}-y_{\boldsymbol j})\mu(y_{\boldsymbol j}).
\]
In the numerical experiments, the measurement is taken as 
\[
\mathcal M_{\rm num}(x_{\boldsymbol i},k):=64k^4\frac1P\sum_{\omega=1}^P\left[(\Re[u_{\rm num}(x_{\boldsymbol i},\omega,k)]^2-(\Im[u_{\rm num}(x_{\boldsymbol i},\omega,k)]^2)\right],
\]
where $P=1000$ denotes the number of sample paths used to approximate the expectation involved in $\mathcal M(x_{\boldsymbol i},k)$. 
Then the strength $\mu$ at the discrete points $\{y_{\boldsymbol j}\}_{\boldsymbol j\in\mathcal T_D}$ can be numerically recovered through the formula
\begin{align}\label{eq:num}
\mathcal M_{\rm num}(x_{\boldsymbol i},k)=\sum_{\boldsymbol j\in\mathcal T_D}|I_{\boldsymbol j}|G(x_{\boldsymbol i}-y_{\boldsymbol j})\mu(y_{\boldsymbol j}).
\end{align}

As mentioned in Section \ref{sec:5.1}, a regularization technique is required to overcome the instability of the inverse problem. Next we introduce the regularized Kaczmarz method, which is a regularized iterative method with two loops.
To enhance the stability and get more accurate reconstructions, we choose $N=4$ measurement domains 
\begin{align*}
U_1&=[1.5,2.5]\times[1.5,2.5],\\
U_2&=[1.5,2.5]\times[-2.5,-1.5],\\
U_3&=[-2.5,-1.5]\times[-2.5,-1.5],\\
U_4&=[-2.5,-1.5]\times[1.5,2.5].
\end{align*}
For each domain $U_n$, $n=1,\cdots,N$, according to \eqref{eq:num}, we get a linear system in the form
\[
b_n=A_nq,\quad n=1,\cdots,N,
\]
where $b_n$ is the discrete measurement vector with components $\mathcal M_{\rm num}(x_{\boldsymbol i},k)$ for $x_{\boldsymbol i}\in U_n$,  $A_n$ is the matrix generated by $G(x_{\boldsymbol i}-y_{\boldsymbol j})$ and $q$ is the unknown vector consisting of $\mu(y_{\boldsymbol j})$. The inner loop of the Kaczmarz algorithm is formed by taking iterations with respect to the index $n$. 
 The outer loop with respect to the index $l=1,\cdots,L$ is used to ensure the convergence of the method as $L\to\infty$ (cf. \cite{BCL17}).
Given an initial guess $q^0=0$, for each $l\in\mathbb N_+$, the regularized Kaczmarz algorithm reads 
\begin{align*}
\begin{cases}
q_0=q^l,\\
q_n=q_{n-1}+A_n^\top(\gamma I+A_nA_n^\top)^{-1}(b_n-A_nq_{n-1}),\quad n=1,\cdots,N,\\
q^{l+1}=q_N,
\end{cases}
\end{align*}
where $\gamma>0$ is the regularization parameter.

\subsection{The numerical examples}

We present two numerical examples to illustrate the validity and effectiveness of the proposed method. 

\begin{figure}[h]
\centering
\subfigure[]{
\includegraphics[width=0.45\textwidth]{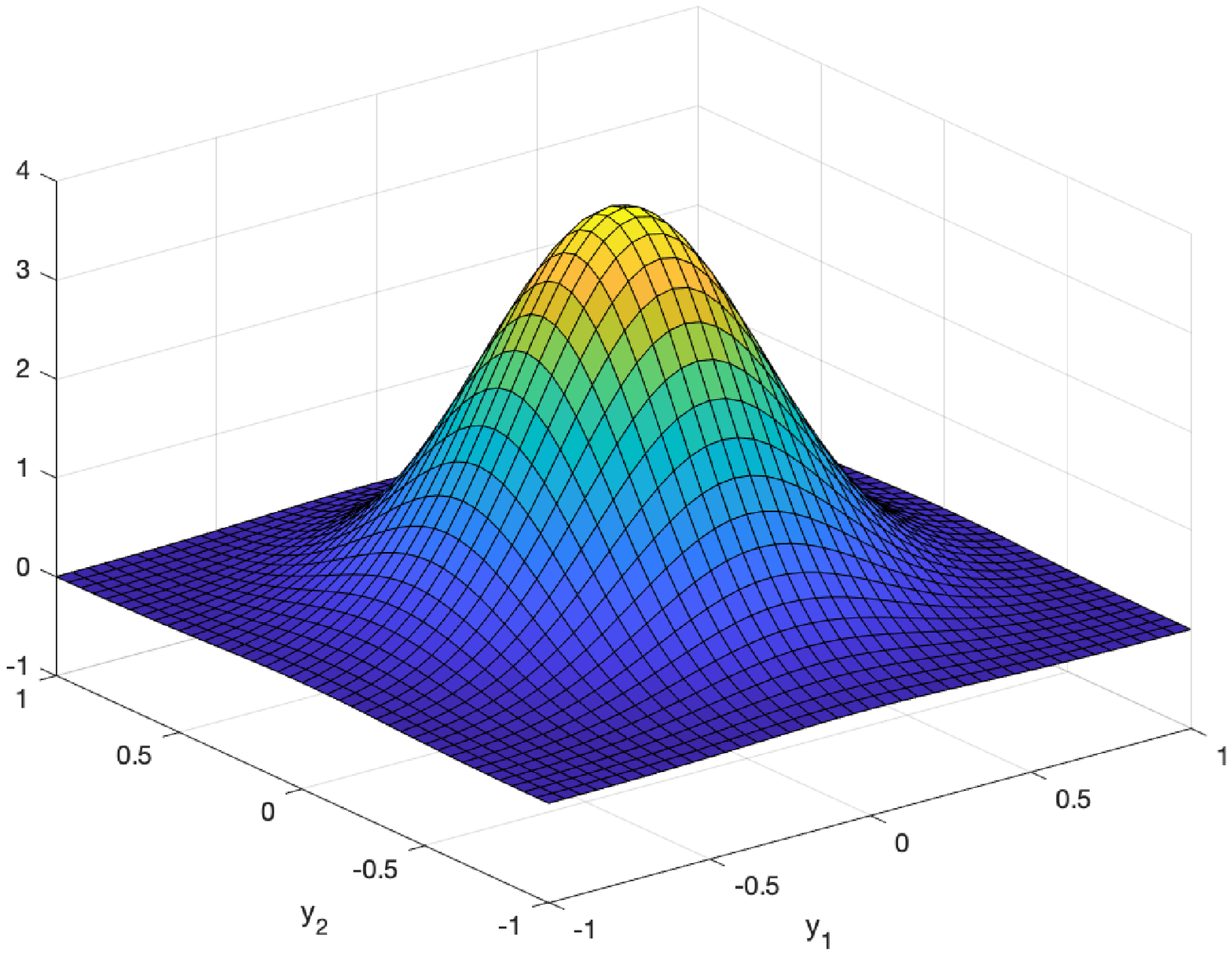}
}
\subfigure[]{
\includegraphics[width=0.45\textwidth]{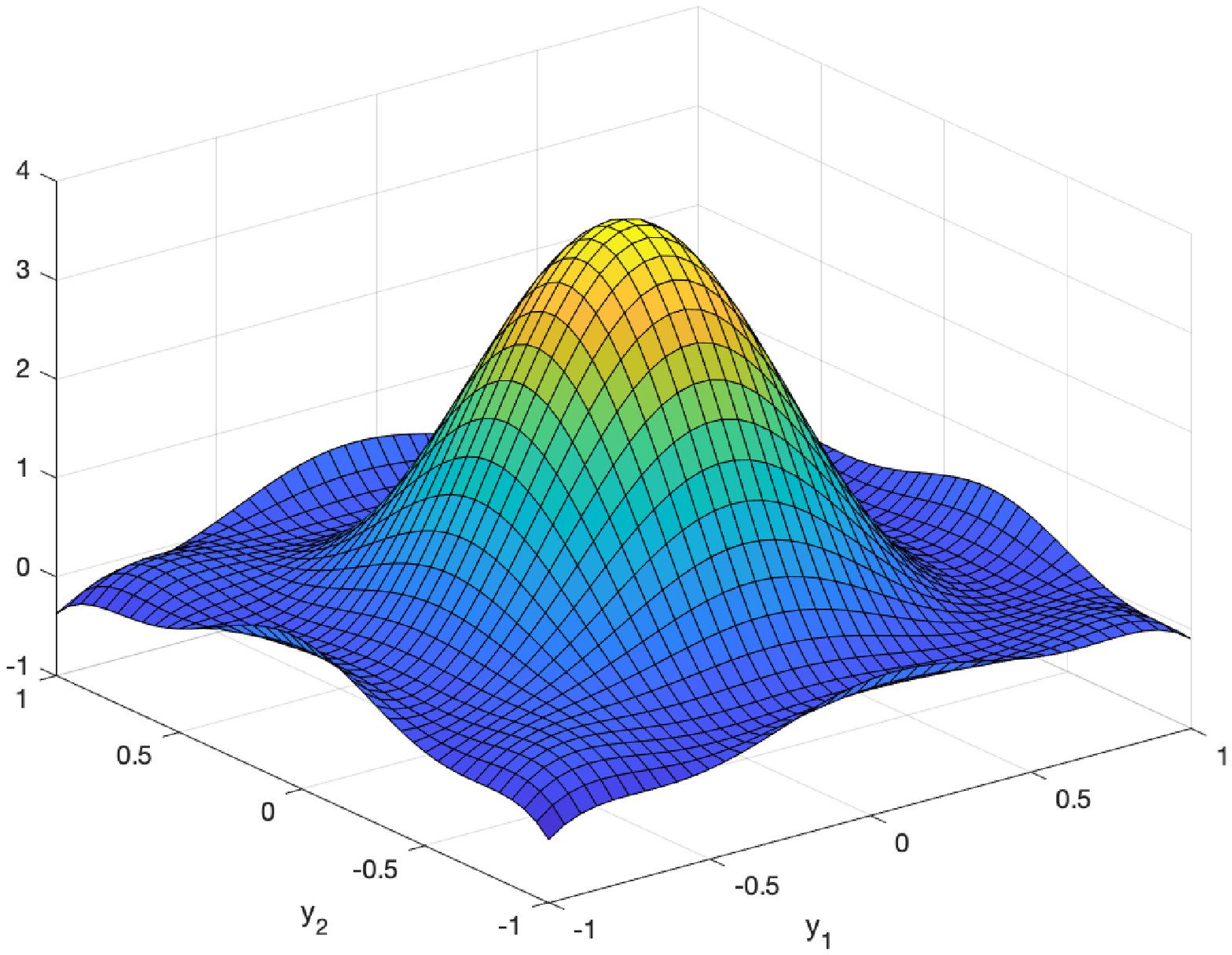}
}
\caption{Example 1: (a) the exact strength; (b) the reconstructed strength at a single frequency $k=2$. }
\label{fig1}
\end{figure}

Example 1: Reconstruct the strength function given by
\[
\mu(y_1,y_2)=4e^{-4(y_1^2+y_2^2)},\quad y=(y_1,y_2)^\top\in D.
\]
The exact strength $\mu$ is plotted in Figure \ref{fig1}(a). For the reconstruction, we choose the iteration number of the outer loop $L=6$ and the regularization parameter $\gamma=10^{-7}$. Figure \ref{fig1}(b) plots the reconstructed strength by using a single frequency $k=2$. It is clear to see that the bump of the exact strength is well reconstructed by using data with only one frequency. The reason is that the strength function only contains few low frequency Fourier modes and all the high frequency Fourier modes decay exponentially fast. 

Example 2: Reconstruct the strength function give by 
\[
\mu(y_1,y_2)=\tilde\mu(3y_1,3y_2), 
\]
where
\[
\tilde\mu(y_1,y_2)=0.3(1-y_1)^2e^{-y_1^2-(y_2+1)^2}-(0.2y_1-y_1^3-y_2^5)e^{-y_1^2-y_2^2}-0.03e^{-(y_1+1)^2-y_2^2}.
\]
The exact strength $\mu$ is plotted in Figure \ref{fig2}(a). This example is harder since it contains a few more Fourier modes than Example 1. It is expected that the multi-frequency data is needed to reconstruct the strength. To incorporate the data with multiple frequencies, one more outer loop is added to the Kaczmarz algorithm and this loop is taken with respect to the wavenumber $k$. We choose the iteration number of the intermediate loop $L=6$ and the regularization parameter $\gamma=10^{-5}$. As a comparison, Figure \ref{fig2}(b) shows the reconstruction at a single frequency $k=2$. Clearly, it is insufficient to reconstruct all the details of the true strength by using only a single low frequency data. Figure \ref{fig2}(c) plots the reconstruction by using multi-frequency data at $k=1:3$. The improvement of the reconstruction is obvious and some details of the true strength are already recovered. Figure \ref{fig2}(d) shows the reconstruction by using multi-frequency data at $k=1:5$. As a few more high frequency data is used, it can be seen that almost all the details of the exact strength are recovered.  

\begin{figure}[h]
  \centering
  \subfigure[]{
  \includegraphics[width=0.45\textwidth]{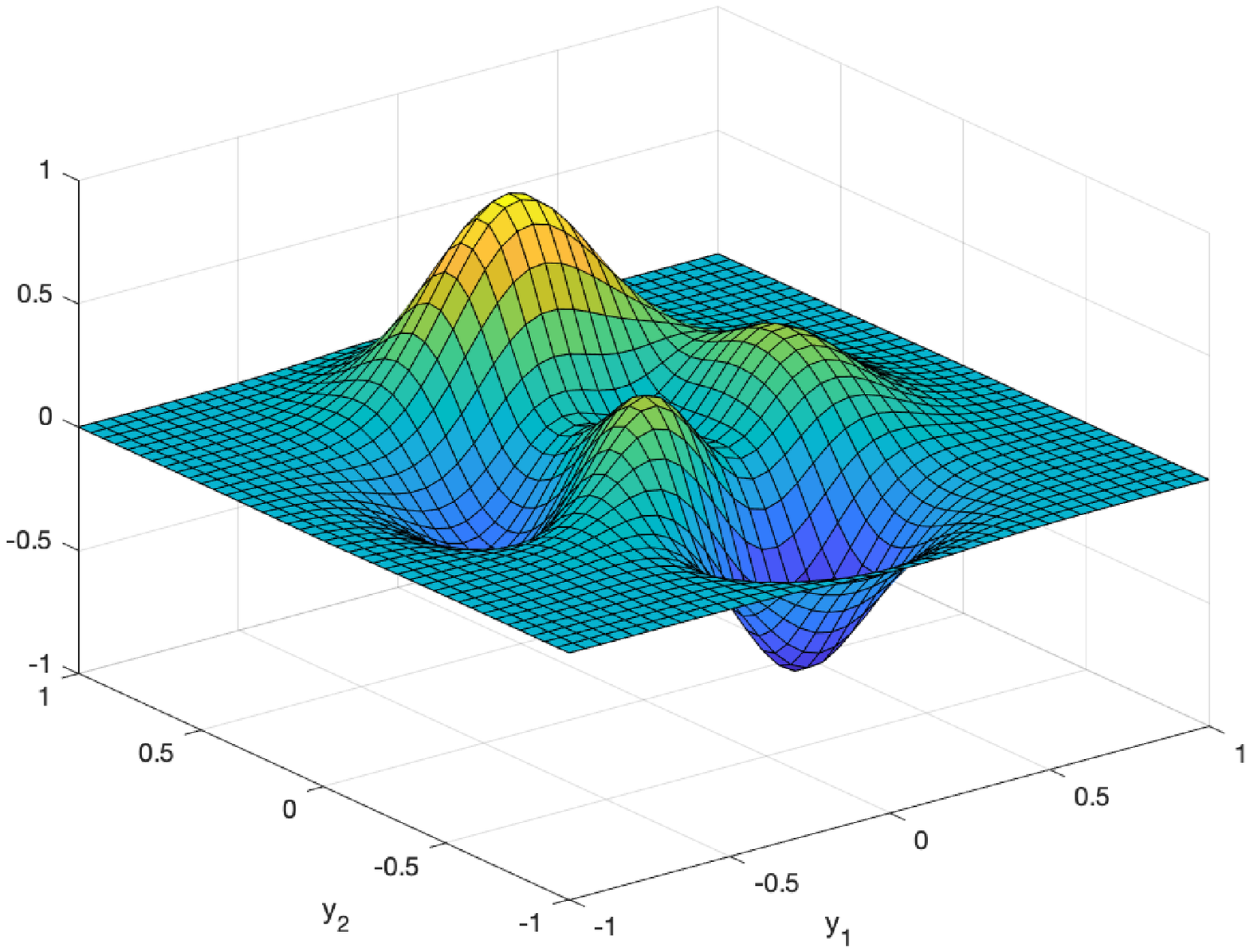}}
  \subfigure[]{
  \includegraphics[width=0.45\textwidth]{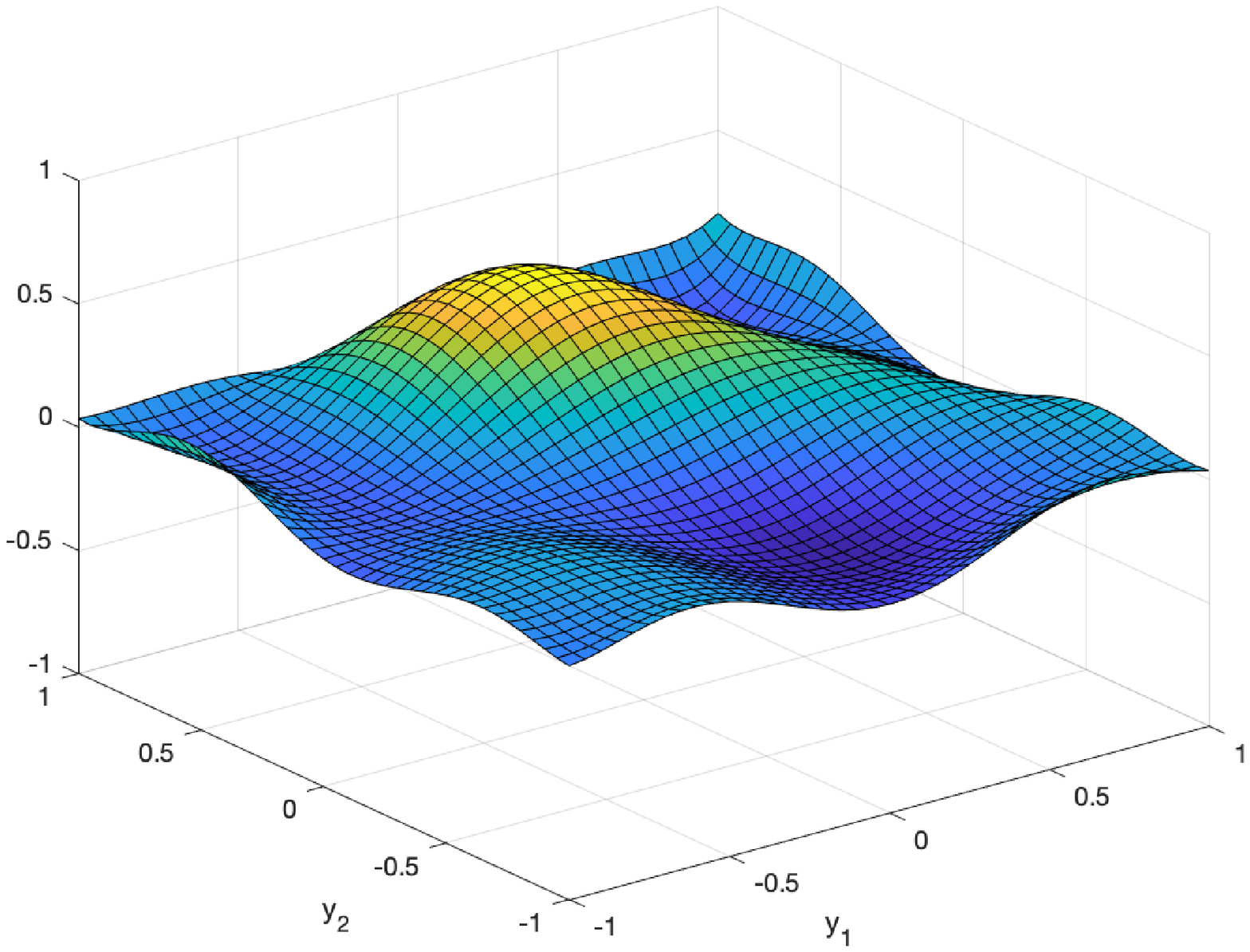}}
  \subfigure[]{
  \includegraphics[width=0.45\textwidth]{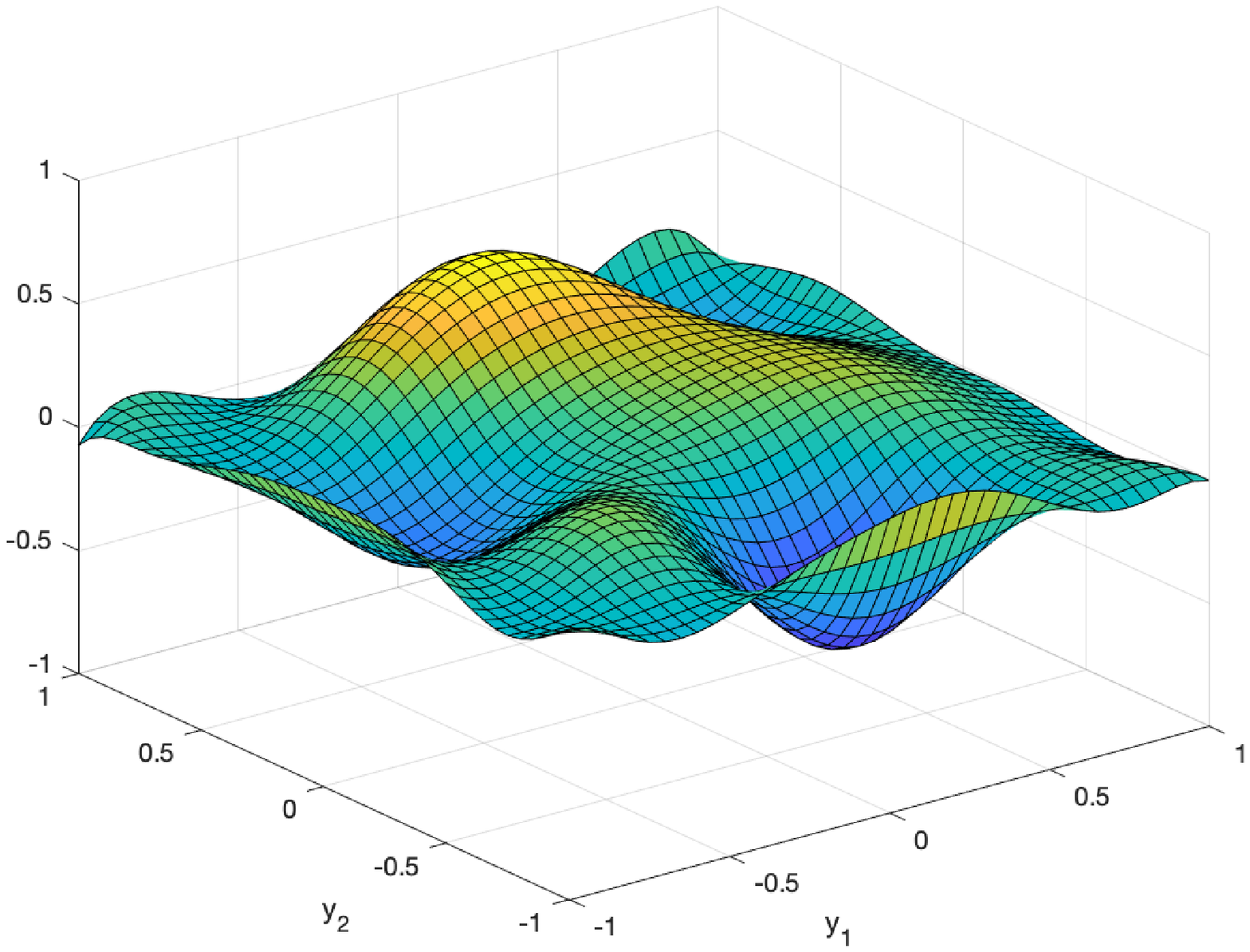}}
  \subfigure[]{
  \includegraphics[width=0.45\textwidth]{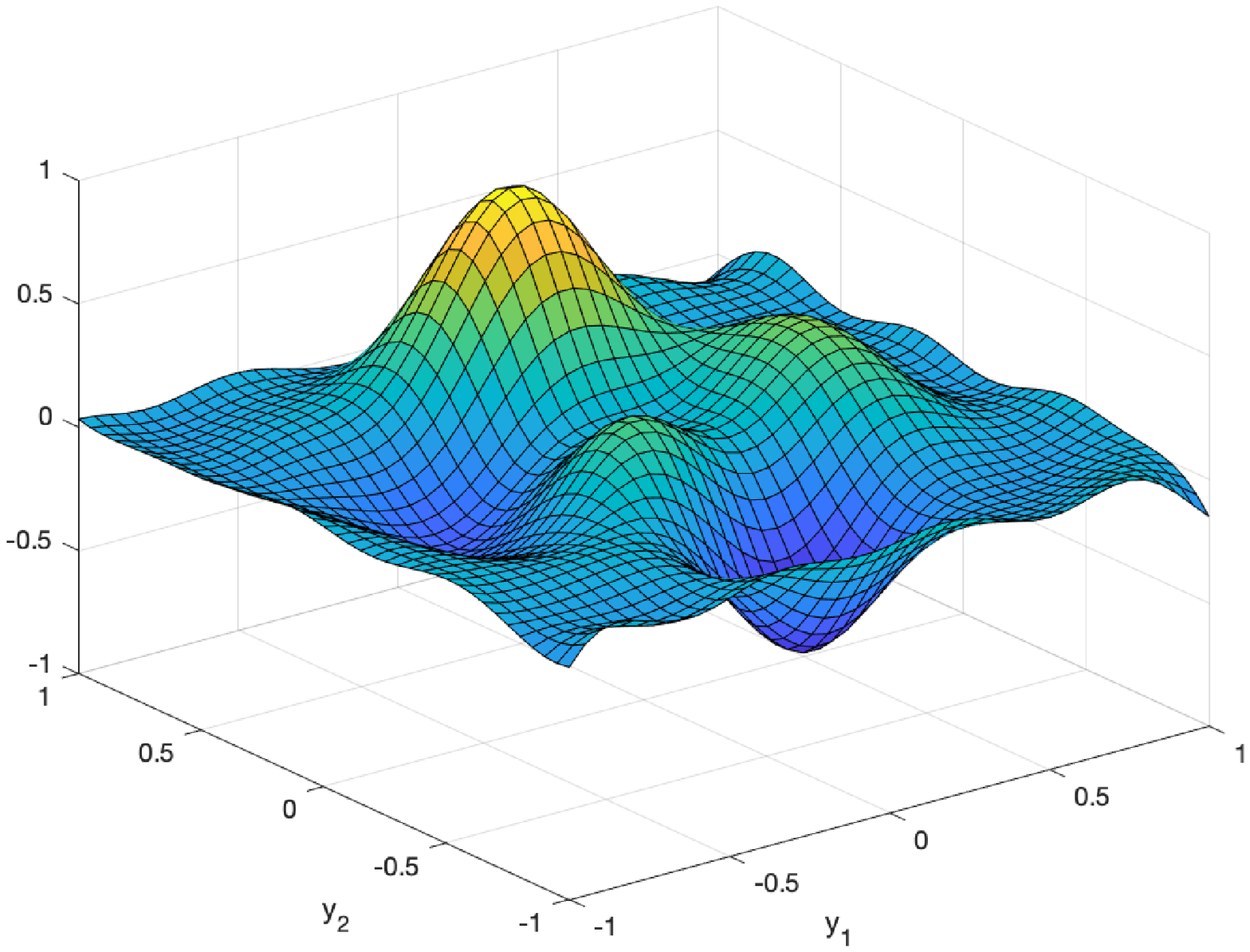}}
   \caption{Example 2: (a) the exact strength; (b) the reconstructed strength at a single frequency $k=2$; (c) the reconstructed strength by using multiple frequencies $k=1:3$; (d) the reconstructed strength by using multiple frequencies $k=1:5$. }\label{fig2}
\end{figure}

\section{Conclusion}\label{co}

We have studied the direct and inverse problems for the stochastic biharmonic operator wave equation driven by a microlocally isotropic Gaussian random source whose covariance operator is a classical pseudo-differential operator. Since the source is too rough to exist pointwisely, it can only be interpreted as a distribution. The well-posedness of the direct problem is obtained in the distribution sense for such a rough source. For the inverse problem, we show that a single realization of the magnitude of the wave field averaged over the frequency band is enough to uniquely determine the strength of the random source. Numerical experiments are presented for the white noise model to demonstrate the effectiveness of the proposed method. 

Since the inverse source problem is linear, one can get an explicit integral expression of the wave field by using the fundamental solution, which is essential in getting the reconstruction formula. If the medium or potential function is a random field, the framework used in the present work is not applicable anymore since the inverse random potential or medium problem is nonlinear. We refer to \cite{KLG12,Y14,TH17,TS18} for related inverse potential problems for the deterministic equations with 
the biharmonic operator. It is open for the inverse random potential or medium problem of the biharmonic wave equation. We hope to be able to report the progress on these problems elsewhere in the future.


\begin{thebibliography}{}

\bibitem{AS92} 
M. Abramowitz and I. A. Stegun, editors, Handbook of Mathematical Functions with Formulas, Graphs, and Mathematical Tables, Dover Publications, Inc., New York, 1992. 

\bibitem{BCL16}
G. Bao, C. Chen, and P. Li, Inverse random source scattering problems in several dimensions, SIAM/ASA J. Uncertainty Quantification, 4 (2016), 1263--1287.

\bibitem{BCL17}
G. Bao, C. Chen, and P. Li, Inverse random source scattering for elastic waves, SIAM J. Numer. Anal., 55 (2017), 2616--2643.

\bibitem{BCLZ14}
G. Bao, S.-N. Chow, P. Li, and H. Zhou, An inverse random source problem for the Helmholtz equation, Math. Comp., 83 (2014), 215--233.

\bibitem{BLLT15}
G. Bao, P. Li, J. Lin, and F. Triki, Inverse scattering problems with multi-frequencies, Inverse Problems, 31 (2015), 093001.

\bibitem{CHL19}
P. Caro, T. Helin, and M. Lassas, Inverse scattering for a random potential, Anal. Appl. (Singap.),
17 (2019), 513--567.

\bibitem{GGS10}
F. Gazzola, H.-C. Grunau, and G. Sweers, Polyharmonic Boundary Value Problems, Positivity Preserving and Nonlinear Higher Order Elliptic Equations in Bounded Domains, Lecture Notes in Mathematics, Springer-Verlag, Berlin, Heidelberg, 2010.

\bibitem{GX19}
Y. Gong and X. Xu, Inverse random source problem for biharmonic equation in two dimensions, Inverse Probl. Imaging, 13 (2019), 635--652.

\bibitem{HLO14}
T. Helin, M. Lassas, and L. Oksanen, Inverse problem for the wave equation with a white noise source, Comm. Math. Phys., 332 (2014), 933--953.

\bibitem{HOUZ10} 
H. Holden, B. \O ksendal, J. Ub\o e, and T. Zhang, Stochastic Partial Differential Equations, A Modeling, White Noise Functional Approach, Second edition, Universitext, Springer, New York, 2010.

\bibitem{I90}
V. Isakov, Inverse Source Problems, Mathematical Surveys and Monographs, 34, American Mathematical Society, Providence, RI, 1990.

\bibitem{KLG12}
K. Krupchyk, M. Lassas, and G. Uhlmann, Determining a first order perturbation of the biharmonic operator by partial boundary measurements, J. Funct. Anal., 262 (2012), 1781--1801. 

\bibitem{KLG14}
K. Krupchyk,  M. Lassas,  and G. Uhlmann, Inverse boundary value problems for the perturbed polyharmonic operator, Trans. Amer. Math. Soc., 366 (2014), 95--112.

\bibitem{LD72}
N. S. Landkof, Foundations of Modern Potential Theory, Translated from the Russian by A. P. Doohovskoy, Die Grundlehren der mathematischen Wissenschaften, Band 180, Springer-Verlag, New York-Heidelberg, 1972.

\bibitem{LPS08}
M. Lassas, L. P\"aiv\"arinta, and E. Saksman,  Inverse scattering problem for a two dimensional random potential, Comm. Math. Phys., 279 (2008), 669--703.

\bibitem{LHL20}
J. Li, T. Helin, and P. Li, Inverse random source problems for time-harmonic acoustic and elastic waves, Comm. Partial Differential Equations, 45 (2020), 1335--1380.

\bibitem{LLW}
J. Li, P. Li, and X. Wang, Inverse elastic scattering for a random potential, arXiv:2007.05790.

\bibitem{LW21}
P. Li and X. Wang, Inverse random source scattering for the Helmholtz equation with attenuation, SIAM J. Appl. Math., 81 (2021), 485--506.

\bibitem{LW21c}
P. Li and X. Wang, An inverse random source problem for Maxwell's equations, Multiscale Model. Simul. 19 (2021),  25--45.

\bibitem{LW21b}
P. Li and X. Wang, Regularity of distributional solutions to stochastic acoustic and elastic scattering problems, J. Differential Equations, 285 (2021), 640--662.

\bibitem{LYZ}
P. Li, X. Yao, and Y. Zhao, Stability for an inverse source problem of the biharmonic operator, arXiv:2102.04631.

\bibitem{LSSW16}
A. Lodhia, S. Sheffield, X. Sun, and S. Watson, Fractional Gaussian fields: a survey, Probab. Surv., 13 (2016), 1--56.

\bibitem{M3P09}
N. V. Movchan, R. C. McPhedran, A. B. Movchan, and C. G. Poulton, Wave scattering by platonic grating stacks, Proc. R. Soc. A, 465 (2009), 3383--3400.

\bibitem{RR20}
J. Rousseau and L. Robbiano, Spectral inequality and resolvent estimate for the bi-harmonic operator, J. Eur.Math. Soc., 22 (2020), 1003--1094.

\bibitem{S00}
A. P. S. Selvadurai, Partial Differential Equations in Mechanics 2 The Biharmonic Equation, Poisson’s Equation, Springer-Verlag, Berlin, 2000.

\bibitem{TH17}
T. Tyni and M. Harju, Inverse backscattering problem for perturbations of biharmonic operator, Inverse Problems, 33 (2017), 105002.

\bibitem{TS18}
T. Tyni and V. Serov, Scattering problems for perturbations of the multidimensional biharmonic operator, Inverse Probl. Imaging, 12 (2018), 205--227.

\bibitem{Y14}
Y. Yang, Determining the first order perturbation of a bi-harmonic operator on bounded and unbounded domains from partial data, J. Differential Equations, 257 (2014), 3607--3639. 

\end{thebibliography}
\end{document}